\documentclass{article}   	
\usepackage{geometry}               
\usepackage{graphicx}
\usepackage[usenames,dvipsnames]{xcolor}
\usepackage{amssymb,amsmath,amsthm,amsfonts}

\allowdisplaybreaks

\usepackage[english]{babel}
\usepackage[utf8]{inputenc}
\usepackage[colorlinks]{hyperref}
\hypersetup{linkcolor=blue,citecolor=blue,filecolor=black,urlcolor=blue}

\usepackage{mathtools}
\usepackage{dsfont}

\usepackage[capitalise, nameinlink, noabbrev]{cleveref} 

\usepackage[sc]{mathpazo}

\usepackage{tikz}
\usetikzlibrary{shadings,intersections,patterns}

\usepackage{pgfplots}

\pgfplotsset{compat=1.10}
\usepgfplotslibrary{fillbetween}
\usetikzlibrary{patterns}
\usetikzlibrary{intersections, calc, angles}
\usepackage{tkz-euclide}
\newtheorem{proposition}{Proposition}[section]
\newtheorem{theorem}[proposition]{Theorem}
\newtheorem{corollary}[proposition]{Corollary}
\newtheorem{lemma}[proposition]{Lemma}

\theoremstyle{definition}

\theoremstyle{remark}
\newtheorem{remark}[proposition]{Remark}
\numberwithin{equation}{section}

\newcommand{\eps}{\varepsilon}

\newcommand{\N}{{\mathbb{N}}}

\newcommand{\R}{{\mathbb{R}}}

\newcommand{\loc}{{\mathrm{loc}}}

\DeclareMathOperator{\dist}{dist}

\DeclareMathOperator{\diverg}{div}

\newcommand{\ind}[1]{\chi_{#1}}

\title{Regularity for one-phase Bernoulli problems with discontinuous weights and applications}
\author{Lorenzo Ferreri, Bozhidar Velichkov}

\begin{document}
\maketitle

We study a one-phase Bernoulli free boundary problem with weight function admitting a discontinuity along a smooth jump interface. In any dimension $N\ge 2$, we show the $C^{1, \alpha}$ regularity of the free boundary outside of a singular set of Hausdorff dimension at most $N-3$. In particular, we prove that the free boundaries are  $C^{1, \alpha}$ regular in dimension $N=2$, while in dimension $N=3$ the singular set can contain at most a finite number of points. We use this result to construct singular free boundaries in dimension $N=2$, which are minimizing for one-phase functionals with weight functions in $L^\infty$ that are arbitrarily close to a positive constant.\\%

\noindent
{\footnotesize \textbf{AMS-Subject Classification} 35R35}. 

\tableofcontents

\section{Introduction}\label{section:introduction}
The aim of this paper is to study the regularity of the local minimizers of the one-phase Alt-Caffarelli functional
\begin{equation}\label{e:definition-J-Q-intro}
J_Q(u)=\int_{B_1}|\nabla u|^2+Q(x)1_{\{u>0\}}\,dx\,,
\end{equation}
with a weight function $Q:\R^N\to\R$ which is merely bounded from above and below by positive constants. Precisely, we consider weight functions of the form
\begin{equation}\label{e:definition-Q-intro}
Q(x)=Q_1(x)1_{E}+Q_2(x)1_{\R^N\setminus E}\,,
\end{equation}
where $Q_1$ and $Q_2$ are H\"older continuous functions on $\R^N$ and the the jump set $\partial E$ is a $C^{1,\alpha}$ manifold. As usual, given a nonnegative function $u:B_1\to\R$, $u\in H^1(B_1)$, we will say that $u$ is a minimizer of $J_Q$ in $B_1$ if
\begin{equation}\label{e:minimality-J-with-jump}
J_Q(u)\le J_Q(v)\quad\text{for every}\quad v\in H^1(B_1)\quad\text{such that}\quad u-v\in H^1_0(B_1).
\end{equation}
Our main theorems are the following
\begin{theorem}[Epsilon-regularity]\label{t:main-main}
Let $N\ge 2$ and let $Q,Q_1,Q_2,E$ and $\alpha$ be as above. Suppose that 
$$1-\delta\le Q_1(x)\le 1\quad\text{and}\quad 2\le Q_2(x)\le 2+\delta\quad\text{in}\quad B_1,$$
for some $\delta\in(0,1/2)$.
Let $u:B_1\to\R$ be a non-negative minimizer of $J_Q$ in $B_1$ in the sense of \eqref{e:minimality-J-with-jump}. Then, there is $\eps>0$ depending on $N$, $Q_1$, $Q_2$, $m$ and $E$, such that the following holds:

If $0\in\partial\{u>0\} \cap \partial E \cap B_1$ and $u$ is $\eps$-flat in $B_1$ in the sense that
\begin{equation}\label{e:definition-eps-flat}
\gamma (x\cdot\nu-\eps)^+\le u(x)\le \gamma (x\cdot\nu+\eps)^+\quad\text{for every}\quad x\in B_1,
\end{equation}
where $\gamma\in[1-\delta,2+\delta]$ and $\nu$ is the interior normal to $\partial E$, then in a neighborhood of $0$ the free boundary $\partial \{u>0\}$ is $C^{1,\beta}$ regular manifold, for some $\beta>0$. Moreover,  
$\partial \{u>0\}$ and $\partial E$ have the same tangent plane at $x_0$. 
\end{theorem}

\begin{theorem}[Regularity of the free boundary]\label{t:main-main-2}
There is a universal constant $N^{\ast\ast}\ge 3$ such that the following holds. Let $Q:\R^N\to\R$ be of the form \eqref{e:definition-Q-intro} with jump set $\partial E$ is a $C^{1,\alpha}$ manifold and weight functions $Q_1,Q_2\in C^{0,\alpha}(\R^N)$ which are bounded from above and from below by positive constants. Then, for every non-negative minimizer $u$ of $J_Q$ in $B_1$, we have: 
\begin{enumerate}
\item[\rm(1)] If $N<N^{\ast\ast}$, then the whole free boundary $\partial\{u>0\}\cap B_1$ is $C^{1,\beta}$ regular for some $\beta>0$. 
\item[\rm(2)]  If $N=N^{\ast\ast}$, then there is a locally finite set $\textrm{Sing}(u)\subset \partial\{u>0\}$ such that $\partial\{u>0\}\setminus \textrm{Sing}(u)$ is a $C^{1,\beta}$ regular manifold for some $\beta>0$. 
\item[\rm(3)] if $N>N^{\ast\ast}$, then the free boundary $\partial\{u>0\}\cap B_1$ can be decomposed as the disjoint union
$$\partial\{u>0\}\cap B_1=\textrm{Sing}(u)\cup \textrm{Reg}(u)\,,$$
where $\textrm{Sing}(u)$ is a closed set of Hausdorff dimension at most $N-N^{\ast\ast}$, while $\textrm{Reg}(u)$ is a $C^{1,\beta}$ regular manifold for some $\beta>0$.
\end{enumerate}
\end{theorem}

\begin{remark}
The critical dimension $N^{\ast\ast}$, from Theorem \ref{t:main-main-2}
 above, is defined as follows: it is the smallest dimension $N\ge 2$ for which there is a nonnegative $1$-homogeneous Lipschitz function $u:\R^{N}\to\R$ whose free boundary $\partial\{u>0\}$ is NOT a regular manifold in the neighborhood of $0$, and which minimizes in $B_1$ a functional $J_Q$ with 
 $Q_1=const$, $Q_2=const$ and $E=\{x_N>0\}$.
 \end{remark}

\begin{remark}
We recall that the critical dimension $N^\ast$ for the classical one-phase Bernoulli problem is defined as follows: it is the smallest dimension $N\ge 2$ for which there is a nonnegative $1$-homogeneous Lipschitz function $u:\R^{N}\to\R$ whose free boundary $\partial\{u>0\}$ is NOT a regular manifold in the neighborhood of $0$, and which minimizes in $B_1$ a functional $J_Q$ with 
 $Q_1=Q_2=const$ and $E=\{x_N>0\}$. Clearly, we have that $N^{\ast\ast}\le N^{\ast}$, but it is not known if the inequality is strict. 
 \end{remark}

\begin{remark}[On the existence of singular cones]
In dimension $N=2$, we show that any one-homogeneous global minimizer $u$ of $J_Q$ (with $Q_1$ and $Q_2$ constants such that $Q_1>Q_2$, and $\partial E$ a plane) is of the form 
$$u(x)=c(x\cdot\nu)_+$$
with $\nu$ the exterior normal to $E$ and $c\in[Q_2,Q_1]$. At the moment we are not able to exclude the existence of cones with singularity on the jump interface, even in dimension $N=3$.
\end{remark}

Our analysis also allows to give a more precise description of the behavior of the free boundaries in dimension two, which we summarize in the following proposition.
\begin{corollary}\label{corollary:2D}
Let $Q:\R^2\to\R$ be of the form \eqref{e:definition-Q-intro} with jump set $\partial E$ is a $C^{1,\alpha}$ manifold and weight functions $Q_1,Q_2\in C^{0,\alpha}(\R^2)$ which are bounded from above and from below by positive constants. Then, the free boundary $\partial\{u>0\}\cap B_1$ of any nonnegative minimizer $u$ (of $J_Q$ in $B_1$) is $C^{1,\beta}$ regular for some $\beta>0$. Moreover, for any boundary point on the jump part of the interface $\partial E$
$$x_0\in\partial\{u>0\}\cap\partial E\cap B_1\qquad\text{with}\qquad Q_1(x_0)<Q_2(x_0),$$
we have that: 
\begin{enumerate}
\item[(i)] in a neighborhood of $x_0$ the free boundary $\partial\{u>0\}$ is the graph of a $C^{1,\alpha}$ function in the direction of the exterior normal $\nu_E(x_0)$ to $E$; 
\item[(ii)] $\partial\{u>0\}$ and 
$\partial E$ are tangent at $x_0$ and the normal $\nu_E$ is pointing outwards $\{u>0\}$ and $E$;
\item[(iii)] if the free boundary $\partial\{u>0\}$ crosses $\partial E$, that is, if for some $r>0$ we have
$$\partial\{u>0\}\cap E\cap B_r(x_0)\neq\emptyset\qquad\text{and}\qquad \partial\{u>0\}\cap (\R^N\setminus \overline E)\cap B_r(x_0)\neq\emptyset\,,$$
then it also sticks to $\partial E$, that is the set $\partial\{u>0\}\cap\partial E\cap B_r(x_0)$ has non-empty interior in $\partial E$.  
\end{enumerate}
\end{corollary}

\subsection{Singular free boundaries in dimension two}
Let $Q:\R^N\to\R$ be a measurable function bounded from above and below by positive constants. It is a long-standing open problem to determine if these minimal assumptions on $Q$ are sufficient in order to establish  the regularity of the free boundary of the minimizers of the functional $J_Q$ analogously to the classical results of Alt and Caffarelli \cite{AltCaffarelli:OnePhaseFreeBd}. Indeed, apart from being a natural generalization of \cite{AltCaffarelli:OnePhaseFreeBd} this question is also related to the study of quantitative geometric functional inequalities involving the Fraenkel asymmetry (e.g. the quantitative Faber-Krahn inequality \cite{BrascoDePhilippsVelichkov2016:SharpFaberKrahn}). \medskip

A consequence of our results, in particular of Corollary \ref{corollary:2D}, is that there are minimizing free boundaries admitting singularities even in dimension two, contrary to what happens in the classical one-phase setting, where the free boundaries are always smooth in 2D. In order to see this, we first notice that an \emph{improvement-of-flatness} cannot hold in such generality (see Remark \ref{remark:counter-example-improvement-of-flatness}). We then build up on this example to show the existence of a singular free boundary in dimension two.

\begin{remark}[On the non-existence of improvement-of-flatness theorems for minimizers of $J_Q$]\label{remark:counter-example-improvement-of-flatness}
Theorem \ref{t:main-main-2} implies that, given a general weight $Q\in L^\infty(\R^N)$ bounded from above and below by positive constants, it is not possible to prove an improvement of flatness result of the form: 
\begin{center}
\it If $u$ minimizes $J_{Q}$ in $B_r$ and $u$ is $\eps$-flat in $B_r$,\qquad \qquad \qquad \qquad\\ \qquad\qquad\qquad\qquad then $u$ is $c\eps$-flat in $B_{r/2}$ for some universal constant $c\in(0,1)$. 
\end{center}
\noindent \rm Indeed, let $N=2$ and 
$$Q_\eps(x):=(1+\eps)1_{\{x_1>0\}}+(1-\eps)1_{\{x_1<0\}}\,.$$
%
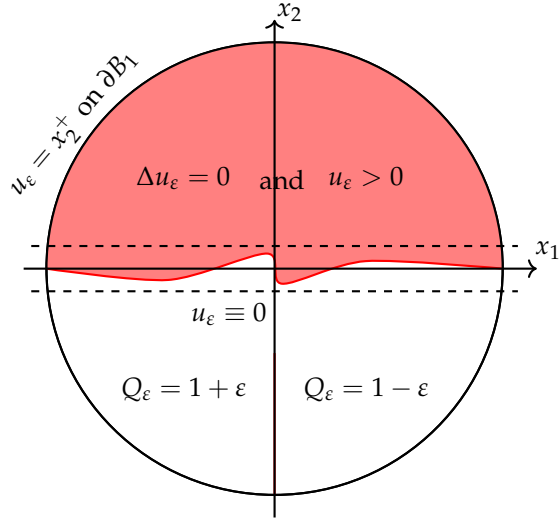
\begin{figure}[h]
\centering{
\begin{tikzpicture}
\coordinate (O) at (0,0);
\draw[] (0,0) circle [radius = 30mm];
\draw[thick] (O) circle [very thick, radius=3cm, name path=c];
\draw [thick, color=red, name path=2] plot [smooth] coordinates  {(0,3)  (0,-3)};
\draw [draw=none, color=red, name path=hor2] plot [smooth] coordinates  {(-1.65,1.12) (1.6,1.12)};
\draw [draw=none, color=red, name path=hor3] plot [smooth] coordinates  {(-1.65,-1.12) (1.6,-1.12)};
\draw [thick, color=red, name path=5] plot [smooth] coordinates {(3,0)  (1.2,0.1) (0.1,-0.2) (-0.1,0.2)  (-1.4,-0.15) (-3,0)};
\begin{scope}[]
\draw[draw=none, fill=red!50] (0,0) -- +(0:3cm) arc (0:180:3cm);
\tikzfillbetween[of=5 and hor2] {color=red!50};
\tikzfillbetween[of=5 and hor3] {color=white};
\end{scope}
\draw [thick, color=red] plot [smooth] coordinates {(3,0)  (1.2,0.1) (0.1,-0.2) (-0.1,0.2)  (-1.4,-0.15) (-3,0)};
\draw[thick] (O) circle [very thick, radius=3cm];
\node[label={[rotate=0]:$u_\eps>0$}] at (1.2,0.8) {};
\node[label={[rotate=0]:$\Delta u_\eps=0$}] at (-1.2,0.8) {};
\node[label={[rotate=0]:and}] at (0.1,0.8) {};
\node[label={[rotate=0]:$u_\eps\equiv0$}] at (-0.6,-1) {};
\node[label={[rotate=50]:$u_\eps=x_2^+$ on $\partial B_1$}] at (-2.4,1.6) {};
\draw node at (-1.2,-1.6) {$Q_\eps=1+\eps$};
\draw node at (1.2,-1.6) {$Q_\eps=1-\eps$};
\draw[thick, ->] (0,-3.3) -- (0,3.3);
\draw[thick, ->] (-3.3,0) -- (3.45,0);
\draw[thick, dashed] (-3.2,0.3) -- (3.2,0.3);
\draw[thick, dashed] (-3.2,-0.3) -- (3.2,-0.3);
\draw node at (0.2,3.4) {$x_2$};
\draw node at (3.6,0.25) {$x_1$};
\end{tikzpicture}
}
\caption{A sketch of the construction in Remark  \ref{remark:counter-example-improvement-of-flatness}. \label{fig:NoEpsilonRegularityThm}}
\end{figure}
%
Take a minimizer $u_\eps$ of $J_{Q_\eps}$ in $B_1$ with $u_\eps=(x_2)^+$ on $\partial B_1$, then $u_\eps$ is $C\eps$-flat in $B_1$. As $\eps\to0$, the minimizers $u_\eps$ converge uniformly to the minimizer $u_0$ of the Alt-Caffarelli functional
$$J(u)=\int_{B_1}|\nabla u|^2\,dx+|\{u>0\}\cap B_1|,$$
for the boundary datum $(x_2)^+$, which is known to be exactly $u_0=(x_2)^+$. By a sliding viscosity argument, it is immediate to check that $u_\eps$ should be $2\eps$-flat in the direction $e_2$. Precisely,
$$\frac{1}{1-\eps}(x_2-\eps)_+\le u_\eps(x_1,x_2)\le \frac{1}{1+\eps}(x_2+\eps)_+\qquad\text{for}\qquad (x_1,x_2)\in B_1.$$ 
On the other hand, by Corollary \ref{corollary:2D}, the free boundaries $\partial\{u_\eps>0\}$ are crossing tangentially $\partial E=\{x_1=0\}$ so that the oscillation of the normal (from scale $1$ to scale $0$) is much larger than the initial flatness. 
\end{remark}

Using the idea from the previous remark and the results from Corollary \ref{corollary:2D}, we obtain the following result. 

\begin{corollary}
For every $\eps>0$ there is a measurable function $Q:\R^2\to\R$ with
$$1-\eps\le Q(x_1,x_2)\le 1+\eps\quad\text{for every}\quad (x_1,x_2)\in\R^2,$$
and a function $u:\R^2\to\R$ which is a minimizer of $J_Q$ in $B_1$ satisfying 
$$(x_2-\eps)_+\le u(x_1,x_2)\le (x_2+\eps)_+\qquad\text{for}\qquad (x_1,x_2)\in B_1,$$
such that $0\in\partial\{u>0\}$ is a singular free boundary point. 
\end{corollary}
\begin{proof}
We first notice that, reasoning as in Remark \ref{remark:counter-example-improvement-of-flatness}, we have that for every $\eps>0$ there is $\delta>0$ with the following property: for any $Q:B_1\to\R$ satisfying
$$1-\delta\le Q(x_1,x_2)\le 1+\delta\quad\text{for every}\quad (x_1,x_2)\in\R^2,$$
and any minimizer $u_Q:B_1\to\R$ of $J_Q$ in $B_1$ with boundary datum
$$u_Q=x_2^+\quad\text{on}\quad \partial B_1,$$
we have 
$$(x_2-\eps)_+\le u_Q(x_1,x_2)\le (x_2+\eps)_+\qquad\text{for}\qquad (x_1,x_2)\in B_1.$$
Take 
$$Q(x_1,x_2):=\begin{cases}
1-\delta & \text{if}\quad \frac{1}{n}<|x_1|<\frac{1}{n+1}\quad\text{with } n \text{ even};\\
1+\delta & \text{if}\quad \frac{1}{n}<|x_1|<\frac{1}{n+1}\quad\text{with } n \text{ odd}.
\end{cases}$$
Notice that the free boundary $\partial \{u>0\}$ has to cross all the vertical lines $\{x_1=\pm\frac1n\}$ in $B_1\cap\{-\eps<x_2<\eps\}$ and that the exterior normal to $\{u>0\}$ points to the right when the free boundary crosses $\{x_1=\frac1n\}$ with $n$ even, and points to the left when the free boundary crosses $\{x_1=\frac1n\}$ with $n$ odd (the picture is reflected when $\{x_1=-\frac1n\}$). Moreover, at every crossing, $\{u>0\}$ sticks to the jump set $\{x_1=\pm\frac1n\}$ in the sense described in Corollary \ref{corollary:2D} (iii). 

Finally, take a point $x_\infty=(0,y)$, obtained as limit of points $x_n=(\frac1n,y_n)\in\partial\{u>0\}$. Then, $x_\infty\in\partial\{u>0\}$, but the free boundary $\partial\{u>0\}$ is not locally the graph of a continuous function. Now, it is sufficient to take $x_\infty$ as center of the coordinate system and to rescale the picture by $1-|y|$.
\end{proof}

\subsection*{Outline of the proofs and plan of the paper}
The main result of the paper is Theorem \ref{t:main-main}. Indeed, Theorem \ref{t:main-main-2} follows Theorem \ref{t:main-main} and the Federer's dimension reduction principle from which we deduce the dimension bounds on the singular set. Corollary \ref{corollary:2D} then follows from Theorem \ref{t:main-main-2}, Theorem \ref{t:main-main} and Lemma \ref{l:regular-blow-ups}. \medskip
 
 In order to prove Theorem \ref{t:main-main}, we first notice that inside the open sets $E$ and $\R^d\setminus \overline E$ any minimizer $u$ of $J_Q$ is a local minimizer of the one-phase problem with H\"older continuous weight function ($Q_1$ and $Q_2$, respectively). Thus, the regularity of the boundary in these sets is a consequence of the classical results of Alt-Caffarelli \cite{AltCaffarelli:OnePhaseFreeBd} (see also \cite{DeSilva:FreeBdRegularityOnePhase} and \cite{Velichkov:RegularityOnePhaseFreeBd}). Thus, we only need to study the behavior of the free boundary at points of the jump set $\partial E$.\medskip

There is a $C^{1,\alpha}$ change of coordinates that transforms $E$ into a half-space, precisely 
$$E=\{x_N>0\}\quad\text{and}\quad \partial E=\{x_N=0\},$$ 
while the functional $J_Q$ in these new coordinates becomes 
$$\int_{B_1}\Big(\nabla u\cdot A(x)\nabla u+Q(x)\ind{\{u>0\}}\Big)\,dx,$$
where $A(x)$ is a matrix with variable coefficients with: 
\begin{equation}\label{e:conditions-on-A-intro}
\quad A(x) \in Sym(n), \quad [A(x)]^{ij} = a^{ij}(x) \in C^{0, \alpha}\left( \overline{B_1} \right),
\end{equation}
that satisfies the following uniform ellipticity condition
\begin{equation}\label{e:conditions-on-A-intro-2}
\exists \, c>0 : c^{-1} \vert \zeta \vert^2 \le a^{ij}(x) \zeta^i \zeta^j \le c \vert \zeta \vert^2, \quad \forall  x \in B_1 .
\end{equation}
The weight $Q$ also changes, but remains H\"older continuous and bounded away from zero: $Q = Q(x) \in C^{0, \alpha}(B_1)$ and there are constants $Q_{min}$ and $Q_{max}$ such that  
\begin{equation}\label{e:conditions-on-Q-intro}
0 < Q_{min} \le Q(x) \le Q_{max} \quad \text{for every}\quad  x \in B_1 .
\end{equation}
Now, the key observation (see Lemma \ref{lemma:ApplQDiscSolFlatAtContact}) is that in a neighborhood of a point $x_0\in\partial E$ of the reduced boundary $\partial^\ast\{u>0\}$ can happen only one of the following: 
\begin{itemize}
    \item[(i)] the free boundary is contained in $\overline E$, that is, $\overline{\{u>0\}} \cap \{ x_N > 0 \} = \emptyset$, \\

    \item[(ii)] the free boundary is contained in $\R^N\setminus E$, that is, $\{ x_N \le 0 \} \subseteq \overline{\{u>0\}}$, \\

    \item[(iii)] the free boundary coincides with $\partial E$, that is, $\overline{\{u>0\}}   = \{ x_N \le 0 \}$. \\
\end{itemize} 

In the case (i) the regularity of the free boundary follows from the work of Chang-Lara and Savin \cite{ChangLaraSavin:BoundaryRegularityOnePhase}. In order to prove Theorem \ref{t:main-main}, we have to deal with case (ii). We do this in Theorem \ref{thm:NonSharpRegularity}, where we prove a regularity result for solutions to general free boundary problems of the following form:
\begin{equation}\label{eqn:MainPb}
\begin{cases}
Lu = f & \text{in } \{ u>0 \} \cap B_1 , \\
u = g & \text{on } \partial B_1 , \\
\vert D_a u \vert = Q & \text{on } \partial\{ u>0 \} \cap B_1^- , \\
\vert D_a u \vert \le Q & \text{on } \partial\{ u>0 \} \cap \{ x_N = 0 \} \cap B_1 , \\
u > 0 & \text{in } B_1^+ ,
\end{cases}
\end{equation}
where $0 \le u \in H^1(B_1)$, $0 \le g\in H^1(B_1)$, the operator $L = L(x)$ is defined as 
\begin{equation}\label{e:definition-L}
L(x)u = \diverg(A(x) \nabla u),
\end{equation}
with $A$ and $Q$ as above. We also set $\vert D_a u \vert \coloneqq \left( a^{ij} \partial_i u \partial_j u \right)^{1/2}$, while for the right hand side, we require that $f \in L^{\infty}(B_1)$ with $f \le 0$ on $B_1^+$ (we notice that for the application to Theorem \ref{t:main-main}, we only need to study the problem in the case $f\equiv 0$). \medskip

Denoting with $\mathcal{F}$ the functional
\[
\mathcal{F}(u) \coloneqq \int_{B_1} a^{ij} \partial_i u \partial_j u + 2 \int_{B_1} f u + \int_{\{u>0\} \cap B_1} Q^2 ,
\]
we define a \emph{(variational) solution} to problem \eqref{eqn:MainPb} any non-negative function $u \in H^1(B_1)$, with $u - g \in H^1_0(B_1)$, that satisfies
\begin{equation}\label{eqn:VariationalSOlution}
\mathcal{F}(u) \le \mathcal{F}(v) , \quad \forall \varphi \in H^1(B_1) \text{ with } v-u \in H^1_0(B_1) \text{ and } B_1^+ \subseteq \{ v > 0\} \cap B_1.
\end{equation}
In the spirit of \cite{ChangLaraSavin:BoundaryRegularityOnePhase}, we prove the following result for variational solutions to \eqref{eqn:MainPb}:
\begin{theorem}\label{thm:NonSharpRegularity}
Let $A(x)$ be a symmetric matrix with variable $C^{0,\alpha}$ coefficients satisfying \eqref{e:conditions-on-A-intro} and \eqref{e:conditions-on-A-intro-2}. Let $Q\in C^{0,\alpha}(B_1)$ be a strictly positive function satisfying \eqref{e:conditions-on-Q-intro}. Let L(x), g(x) and f(x) be as in \eqref{eqn:MainPb} and let u(x) be a variational solution to problem \eqref{eqn:MainPb} in the sense of \eqref{eqn:VariationalSOlution}.

Then, for any $\beta \in [0, \min\{\alpha,1/2\})$, the set $\partial\{ u>0 \} \cap B_1$ is of class $C^{1, \beta}$ in a neighborhood of any free boundary point $x \in \partial{\{u>0\}} \cap \{ x_N = 0 \} \cap B_1$.
\end{theorem}
\begin{remark}
The condition on the non-positivity of $f$ over $D$ is used in Remark \ref{rmk:BoundBelowfNotPositive} and in the proof of Lemma \ref{lemma:LocalNonDegenerate} below. Basically, it allows to compare from below the solutions with positive $L$-harmonic functions over suitable domains.
\end{remark}

\begin{remark}
The same regularity result holds for flat Lipschitz continuous viscosity solutions to the problem \eqref{eqn:MainPb} (see Section  \ref{section:ImprovFlatness}). We also notice that, for the proof of Theorem \ref{thm:NonSharpRegularity}, our general strategy was inspired by the work of Chang-Lara and Savin \cite{ChangLaraSavin:BoundaryRegularityOnePhase}, although the key technical steps are different. 
\end{remark}
\begin{remark}
Using the approach from \cite{ChangLaraSavin:BoundaryRegularityOnePhase} one can show that, when $\alpha\ge 1/2$, Theorem \ref{thm:NonSharpRegularity} holds with $\beta=1/2$. This, in particular, means that the sharp regularity of the free boundary in Theorem \ref{t:main-main} is $C^{1,1/2}$.
\end{remark}
\subsection*{Plan of the paper}
The sections \ref{section:preliminaries}, \ref{section:blowUp}, \ref{section:BoundaryHarnack}, \ref{section:ImprovFlatness}, and \ref{section:Regularity} are dedicated to the proof of Theorem \ref{thm:NonSharpRegularity}. Precisely:
\begin{itemize}
\item In Section \ref{section:preliminaries} we prove the Lipschitz continuity and the nondegeneracy of $u$.
\item In Section \ref{section:blowUp}, we study the blow-up limits of $u$ at free boundary points lying on the hyperplane $\{x_N=0\}$. In the same section, we show that the variational solutions to \eqref{eqn:MainPb} are also viscosity solutions to \eqref{eqn:MainPb}.
\item In Section \ref{section:BoundaryHarnack} we prove a Partial Harnack Inequality for viscosity solutions of \eqref{eqn:MainPb}. 
\item In Section \ref{section:ImprovFlatness} we prove an improvement-of-flatness theorem for flat viscosity solutions of \eqref{eqn:MainPb}. 
\item In Section \ref{section:Regularity} we conclude the proof of Theorem \ref{thm:NonSharpRegularity}.
\end{itemize}
In Section \ref{section:main-main} we prove Theorem \ref{t:main-main}, while Section \ref{section:main-main-2} is dedicated to the proof of Theorem \ref{t:main-main-2}.\medskip

\noindent Finally, in Section \ref{section:applications}, we use Theorem \ref{thm:NonSharpRegularity} to answer a question left open in \cite{BucurButtazzoVelichkov:ShapeOptimizInternalContraint} on the regularity of the solutions to a shape optimization problem with internal constraint.

\section{Preliminaries}\label{section:preliminaries}

In the following, by saying that the origin is a point of detachment for the free boundary, we mean that
\begin{align} \label{eqn:0DetachmentCond}
\begin{split}
& 0 \in \partial \{ u>0 \} \cap \{ x_N = 0 \} \cap B_1, \\
\forall \, r & >0 \, \exists \, x \in B_r : x \in \partial \{ u>0 \}\cap B_r^{-}, \\
\forall \, r>0 \, & \exists \, x \in B_r : x \in \partial \{ u>0 \} \cap \{ x_N = 0 \} \cap B_r.
\end{split}
\end{align}
Throughout the paper, without loss of generality, we assume the following additional hypotheses on the coefficients of $A$, on the weight $Q$ and on the right-hand side $f$:
\begin{align}
& \quad Q(0)=1, \quad a^{ij}(0) = \delta^{ij}, \label{eqn:SmallnessCond1}\\
\| a^{ij}(x) - & \delta^{ij} \|_{C^{0, \alpha}\left( \overline{B_1} \right)} + \|Q(x) - 1 \|_{C^{0, \alpha}\left( \overline{B_1} \right)} + \| f \|_{L^{\infty}\left( \overline{B_1} \right)} \le \varepsilon \delta, \label{eqn:SmallnessCond2}
\end{align}
for some $\varepsilon, \delta > 0$ sufficiently small.\medskip

A key result for the subsequent analysis is the following

\begin{lemma}\label{lemma:linearBehavior}
Let $L = L(x)$ be a uniformly elliptic operator in $B_1^+$, such that

\begin{equation}\label{eqn:HypLinBehavAtBd}
L(x)u = \diverg(A(x) \nabla u), \quad [A(x)]^{ij} = a^{ij}(x) \in C^{0, \alpha}\left( \overline{B_1} \right), \quad \| a^{ij}(x) - \delta^{ij} \|_{C^{0, \alpha}\left( \overline{B_1} \right)} \le \varepsilon ,
\end{equation}
and let $u \in H^1(B_1)$ be a nonnegative continuous function in $B_1$ such that $u>0$ in $B_1^+$, $Lu=0$ in $\{u>0\}$, and $u(0)=0$. Then, there exist constants $\overline{\varepsilon}>0$ and $\theta > 0$ such that, for any $0 < \varepsilon < \overline{\varepsilon}$ one the following properties holds in $B_1^+$:

\begin{itemize}

\item[(i)] $u(x)$ grows more than any linear function in the N-th direction at $x=0$, that is, 
$$\lim_{\begin{array}{cc}|x|\to0\\  x_N\ge \lambda|x|\end{array}}\frac{u(x)}{|x|}=+\infty\quad\text{for every}\quad \lambda\in(0,1).$$

\item[(ii)] $u(x) \ge \theta (1+o(1)) x_N$, and equality holds along any non-tangential direction to the plane $\{ x_N = 0 \}$, that is, for every $\lambda\in(0,1)$ we have 
$$u(x)=\theta(1+o(1)) x_N\quad\text{in the set}\quad \{x_N\ge \lambda|x|\},$$
where the quantity $o(1)$ is intended for $\vert x \vert \to 0$.

\end{itemize}

\end{lemma}

\begin{proof}
We basically adapt the proof of \cite[Lemma 11.17]{CaffarelliSalsa:GeomApproachToFreeBoundary} to the case of elliptic operators in divergence form with $C^{0, \alpha}$ coefficients.

To begin with, consider a rectangular annulus $R \subseteq B_1^+$ with smoothed vertices so that its boundary is of class $C^{\infty}$, with one of the sides parallel to the plane $x_N = 0$ and touching it in an open subset containing the origin (see Figure \ref{fig:LinearBehaviorDomainFig}). Let us denote $\partial R_1, \partial R_2$ its outer and inner boundaries, respectively. Then, let us consider the functions $v, h$ defined in the following way:

\begin{align}\label{eqn:LinBehavCompetitDef}
\begin{split}
& \Delta v = 0 \text{ in } R, \quad v = 0 \text{ on } \partial R_1, \quad v = 1 \text{ on } \partial R_2, \\
& L h = 0 \text{ in } R, \quad h = 0 \text{ on } \partial R_1, \quad h = 1 \text{ on } \partial R_2.
\end{split}
\end{align}

By the elliptic regularity for the laplace operator and the smoothness of $R$, it holds that $v \in C^{\infty}(\overline{R})$. Thus, it can be extended to a smooth function over the whole $\R^N$, and at $x = 0$ it has an asymptitic expansion of the form

\begin{equation}\label{eqn:xNHarmBarrierAsymptExpansion}
v(x) = \gamma (1+o(1)) x_N \quad \text{as } \vert x \vert \to 0,
\end{equation}
where $\gamma > 0$ by Hopf's lemma.

Now we turn to the function $h$. By elliptic regularity (e.g. \cite[Theorem 8.34]{GilbargTrudinger:Elliptic98}) $h \in C^{1, \alpha}\left( \overline{R} \right)$. Hence, the function $z \coloneqq h - v$ belongs to $C^{1, \alpha}\left( \overline{R} \right)$, and solves

\[
Lz = \diverg \left( F \right) \quad \text{in } R, \quad \text{with } F \coloneqq (I - A) \nabla v,
\]
and by \eqref{eqn:HypLinBehavAtBd} it holds that $\| F \|_{C^{0, \alpha}\left( \overline{R} \right)} \to 0$ as $\varepsilon \to 0$. Using the Schauder estimates, this gives $\| z \|_{C^{1, \alpha}\left( \overline{R} \right)} \to 0$ for $\varepsilon \to 0$.

Using \eqref{eqn:xNHarmBarrierAsymptExpansion} and the fact that $z = 0$ on $\{ x_N = 0 \}$ sufficiently near the origin, the above convergence property implies that, if $\varepsilon < \overline{\varepsilon}$ sufficiently small, the function $h$ satisfies
\begin{equation}\label{eqn:xNBarrierAsymptExpansion}
    h(x) = \beta (1+o(1)) x_N  \quad \text{as } \vert x \vert \to 0
\end{equation}
for some $\beta > 0$.

Now we actually proceed to prove $(i)$ and $(ii)$. Let $r_0$ be such that $B_r^+ \subset R$ for all $0<r<r_0$. Then, consider the constants $\alpha_r$ defined as
\[
\alpha_r \coloneqq \sup \{c \in \R^+ \, : \, u(x) \ge c h(x) \text{ in } B_r^+ \}, \quad 0<r<r_0\,,
\]
which are non-increasing in $r$, thus strictly positive by the strong maximum principle and the fact that $h = 0$ on $\partial R_1$. Now define 
\[
\alpha \coloneqq \sup_{0<r<r_0} \alpha_r.
\]

If $\alpha = +\infty$ then $(i)$ is proved, thus we are left to treat the case $\alpha < +\infty$. In such case,

\[
u - \alpha h \ge \frac{1}{\alpha}(\alpha_r - \alpha)  \alpha h \quad \text{in } B_r,
\]
so that \eqref{eqn:xNBarrierAsymptExpansion} implies that

\begin{equation}\label{eqn:xNuAsymptExpansion}
    u(x) \ge \alpha \beta (1+o(1)) x_N  \quad \text{as } \vert x \vert \to 0,
\end{equation}
which gives the first part of $(ii)$. In order to prove that equality holds, we proceed by contradiction. We only treat the case of the direction $e_N$ since, for any other non-tangential direction $\nu$, the result follows in a similar fashion, by composing the analogous of \eqref{eqn:CondForCompPrinciple} below with a smooth diffeomorphism of $B_1^+$ into itself, and sending the direction $\nu$ to $e_N$. 

Suppose that there exists a sequence $x_{N_k}$ and a constant $\delta_0 > 0$ such that
\[
u\left(x_{N_k}\right) - \alpha \beta x_{N_k} \ge  \delta_0 \alpha \beta x_{N_k},
\]
or equivalently (by \eqref{eqn:xNBarrierAsymptExpansion})
\[
u\left(x_{N_k}\right) - \alpha h\left(x_{N_k}\right) \ge  \delta_0 \alpha \beta x_{N_k}.
\]
By the Schauder estimates up to the boundary, we can replace $x_N$ by the function $w$ defined by (actually we should consider a smooth subdomain of $B_1^+$, but we avoid making the notation more complex)
\[
Lw = 0 \text{ in } B_1^+, \quad w = x_N \text{ on } \partial B_1^+
\]
and, up to redefining $\delta_0$, we would still get
\[
u\left(x_{N_k}\right) - \alpha h\left(x_{N_k}\right) \ge  \delta_0 \alpha \beta w\left(x_{N_k}\right).
\]
Now, For each k such that $r_k \coloneqq \vert x_{N_k} \vert \le r_0/2$, the following holds
\begin{align}\label{eqn:CondForCompPrinciple}
\begin{split}
& u(x) - \alpha_{2r_k} h(x) \ge  0 \quad \text{in } B_{2r_k} , \\
& u\left(x_{N_k}\right) - \alpha_{2r_k} h\left(x_{N_k}\right) \ge  \delta_0 \alpha \beta w\left(x_{N_k}\right).
\end{split}
\end{align}
At this point we can apply the comparison principle \cite[Theorem 11.6]{CaffarelliSalsa:GeomApproachToFreeBoundary} (notice that, for the bound from above, it is not important whether the function at the denominator vanishes at $\{ x_N = 0 \}$ or not) to deduce the existence of a constant $c_0$ such that
\[
u(x) - \alpha_{2r_k} h(x) \ge  c_0 \delta_0 \alpha \beta w\left(x\right) \quad \text{in } B_{r_k},
\]
or equivalently
\[
u(x) - \alpha_{2r_k} h(x) \ge  c_0 \delta_0 h(x) \quad \text{in } B_{r_k},
\]
but by \eqref{eqn:xNBarrierAsymptExpansion} this leads to a contradiction on the definition of $\alpha$, for $k$ sufficiently large. This concludes the proof.
\end{proof}

\begin{figure}[t]
\centering

\begin{tikzpicture}
\draw[thick] (-4,0) arc [start angle = 180, end angle = 0,
x radius = 40mm, y radius = 40mm];
\draw[thick] (-4,0) -- (4,0); 
\draw[thick, rounded corners, fill=blue!10] (-1.5, 0) rectangle (1.5, 3) {};
\draw[thick, rounded corners, fill=white] (-0.5, 1) rectangle (0.5, 2) {};
\draw node at (1.1,3.2) {$\partial R_1$};
\draw node at (0.5,2.2) {$\partial R_2$};
\draw node at (-3,1) {$B_1^+$};
\draw[thick, dashed, ->] (0,0) -- (0,4.5);
\draw node at (0.4,4.3) {$x_N$};
\end{tikzpicture}
\caption{A sketch of the domain used in the proof of Lemma \ref{lemma:linearBehavior}. \label{fig:LinearBehaviorDomainFig}}
\end{figure}
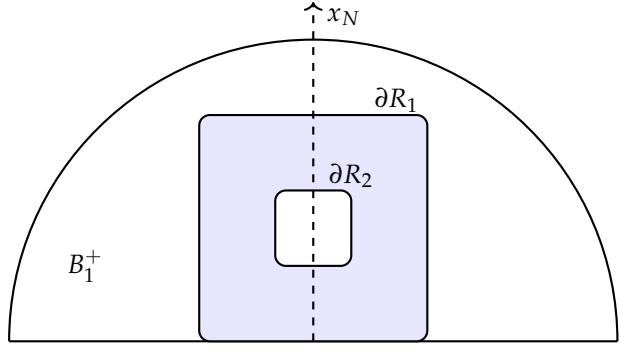

\begin{remark}\label{rmk:BoundBelowfNotPositive}
In the case $Lu=f$ with $f \in L^{\infty}(B_1)$ and $f \le 0$, a bound from below as in point (ii) can still be obtained comparing $u$ with its $L$-harmonic extension near the origin.
\end{remark}

\begin{remark}\label{rmk:scalingHypLinBehavAtBd}
The hypotheses \eqref{eqn:HypLinBehavAtBd} of Lemma \ref{section:preliminaries}, are satisfied by \emph{any} uniform elliptic operator in divergence form with coefficients of class $C^{0, \alpha}$ and satisfying $\quad a^{ij}(0) = \delta^{ij}$, provided one increases sufficiently the scale near the origin. Namely, if the smallness assumptions are substituted by the condition $a^{ij}(0) = \delta^{ij}$, the previous lemma still applies at sufficiently (but finite) small scales near the origin. 
\end{remark}

\begin{remark}\label{rmk:intBallCond}
Suppose that in place of $\overline{B_1^+}$ conditions \eqref{eqn:HypLinBehavAtBd} hold on the closure of an open set $\Omega$ enjoying an interior ball condition, say $B_R(y_0)$ touching $\overline{\Omega}$ at the origin, and let $\nu = e_N$ be the inner normal to such ball at the touching point. Let us also denote with $\Psi(x): B_r \subseteq R^{N-1} \to \R^+$ be a local chart of $\partial B_R(y_0)$ near the origin, with respect to coordinates lying in a plane orthogonal to $\nu$ and such that $\Psi(0) = 0$. Then, extend $\Psi$ constantly along directions parallel to $\nu$, in a whole neighborhood of the origin. After a flattening of $\partial B_R(y_0)$ near the origin, an application of Lemma \ref{lemma:linearBehavior} gives that either $u(x)$ grows more then any linear function, or there exists $\alpha>0$ such that
\[
u(x) \ge \alpha (1+o(1)) \left( x_N - \Psi(x) \right) \quad \text{in } B_R(y_0) \text{ and for } \vert x \vert \to 0, 
\]
with equality holding in any non-tangential direction to $\partial B_R(y_0)$ at the origin. In particular, this implies that
\[
u(x) \ge \alpha x_N + o\left( \vert x \vert \right) \quad \text{in } B_R(y_0) \text{ and for } \vert x \vert \to 0, 
\]
with equality in any non-tangential direction.
\end{remark}

\begin{remark}\label{rmk:intC1aCond}
Analogous considerations to Remark \ref{rmk:intBallCond} still hold true if the inner touching domain is not necessarily a ball but is at least of class $C^{1, \alpha}$.
\end{remark}

Standard properties of variational solutions to problem \eqref{eqn:MainPb} are boundedness, local Lipschitz regularity and non-degeneracy.
\begin{lemma}\label{lemma:LocalBound}
    Let $u$ be a variational solution to \eqref{eqn:MainPb}. Then $u \in L^{\infty}_{\loc}\left( B_1 \right)$.
\end{lemma}
\begin{proof}
    Similarly as in \cite[Lemma 2.7]{Velichkov:RegularityOnePhaseFreeBd}, one can prove that $\diverg\left( A(x) \nabla u \right) \ge f$ in $B_1$ in the sense of distributions. Choose $B_R \subset \subset \Omega$ and let $v \in H^1_0\left( B_R \right)$ be a solution of 
    $$\diverg\left( A(x) \nabla v \right) = -\| f \|_{L^{\infty}(B_1)}\,.$$ Then, the function $u + v$ is $L$-subharmonic in $B_R$, thus by the maximum principle it is uniformly bounded in $B_{R/2}$.
\end{proof}
\begin{lemma}\label{lemma:LocalLipschitz}
    Let $u$ be a variational solution to \eqref{eqn:MainPb}. Then $u \in Lip(\Omega')$ for all $\Omega' \Subset B_1$. 
\end{lemma}
\begin{proof}
    The proof of the Lipschitz regularity follows from the sub-harmonicity of $u$ (see for instance \cite[Lemma 2.7]{Velichkov:RegularityOnePhaseFreeBd}) which allows a laplacian estimate (see e.g. \cite[Lemma 3.9]{Velichkov:RegularityOnePhaseFreeBd}), and a standard pointwise estimate of the gradient (e.g. \cite[Lemma 3.5]{Velichkov:RegularityOnePhaseFreeBd}).
\end{proof}
\begin{lemma}\label{lemma:LocalNonDegenerate}
    Let $u$ be a variational solution to \eqref{eqn:MainPb}. Then, there exists a constant $c>0$ such that
    \[
    \|u\|_{L^{\infty}\left( B_r(x_0) \right)} \ge cr
    \]
    for all $x_0 \in \overline{\Omega^+(u)} \cap B_{1/2}(0)$ and $r>0$ such that $B_r(x_0) \subset B_1(0)$.
\end{lemma}
\begin{proof}
    Let $x_0 \in B_{1/2}(0) \setminus \overline{D}$ and denote $r_0 \coloneqq \dist\left( x_0, \overline{D} \right)$. For $0 < r \le r_0$ the proof follows the lines of \cite{AltCaffarelli:OnePhaseFreeBd}, so that $\|u\|_{L^{\infty}\left( B_r(x_0) \right)} \ge cr$ for some constant $\overline{c}>0$ independent of $x_0$ and $u$.
    
    To extend the non-degeneracy for $r > r_0$ we proceed as follows. We only need to study the case $B_{10 r_0}(x_0) \subset B_1(0)$, since otherwise one can choose $c = \overline{c}/10$. Now distinguish two cases.
    \begin{itemize}
        \item[i)] Let $0 < r \le 10 r_0$. In this case the non-degeneracy holds by choosing again $c = \overline{c}/10$. \\

        \item[ii)] Let $r > 10 r_0$, and denote $y_0 \coloneqq B_{r_0}(x_0) \cap \{ x_N = 0 \}$. In this case the non-degeneracy can be deduced comparing in $B_{r/10}(y_0)$ the function $u$ from below with, for instance, with $L$-harmonic functions in annuli.
    \end{itemize}

    If $x_0 \in \overline{D}$ the non-degeneracy follows again by comparing from below the function $u$ with $L$-harmonic functions in annuli.

\end{proof}

\section{Blow-up and viscosity setting}\label{section:blowUp}
For any variational minimizer of problem \eqref{eqn:MainPb}, under the conditions \eqref{eqn:SmallnessCond1} and \eqref{eqn:SmallnessCond2}, the origin is a regular point of contact between the free boundary $\partial \{ u>0 \} \cap B_1$ and the obstacle $\{x_N \le 0\} \cap B_1$, in the sense that the blow-up at such point is a half plane solution.

Let us denote with $u_r$ the r-scaling of $u$ at the origin, namely
\[
u_r(x) \coloneqq \frac{1}{r} u\left( r x \right), \quad r>0, x \in B_1 ,
\]
and with $\mathcal{F}_{\infty}$ the (limiting) functional defined as
\[
\mathcal{F}_{\infty}(u, \Omega) \coloneqq  \int_{\Omega} \vert \nabla u \vert^2 + \left \vert \{u>0\} \cap \Omega \right \vert ,
\]
then the following lemma holds.

\begin{lemma}\label{lemma:BlowUpConverges}
There exist a sequence $r_n \to 0^+$ and a function $u_{\infty} \in C^0(\R^N) \cap H^1_{loc}(\R^N)$ such that:
\begin{itemize}
\item[i)] $u_{r_n} \to u_{\infty}$ in $L^{\infty}_{loc}(\R^N)$,
\item[ii)] $u_{r_n} \to u_{\infty}$ in $H^1_{loc}(\R^N)$,
\item[iii)] $ \{ u_{r_n}>0 \} \to \{ u_{\infty}>0 \} $ in the sense of $L^1_{loc}(\R^N)$,
\item[iv)] $ \overline{\{ u_{r_n}>0 \}} \to \overline{\{ u_{\infty}>0 \}} $ locally Hausdorff in $\R^N$,
\item[v)] $u_{\infty}$ is a non-tivial variational minimizer of $\mathcal{F}_{\infty}$ in $B_R$, for any $R>0$.
\end{itemize}
\end{lemma}

\begin{proof}
Proceed, for instance, analogously as in \cite[Proposition 6.2]{Velichkov:RegularityOnePhaseFreeBd}.
\end{proof}

Let $u_{\infty}$ as in Lemma \eqref{lemma:BlowUpConverges}, then as a consequence we have the following

\begin{corollary}\label{crl:BlowUpHalfPlane}
There exists a positive sequence $r_n \to 0$ such that $u_{r_n}$, $u_{\infty}$ satisfy $i) - v)$ of Lemma \ref{lemma:BlowUpConverges} and
\[
u_{\infty}(x) = \alpha (x \cdot e_N) \, 1_{\{ x_N \ge 0 \}}
\]
for some constant $\alpha \in(0,1]$.
\end{corollary}

\begin{proof}
Let us denote by $c_n$ the sequence given by Lemma \ref{lemma:BlowUpConverges}, and with $v_{\infty}$ the associated blow-up limit.

By Lemma \ref{lemma:BlowUpConverges} (v) we have that $v_{\infty} \ge 0$ is locally Lipschitz continuous in $\R^N$ and is harmonic in $\{ v_{\infty} > 0 \}$, which is open. Moreover, by Lemma \ref{lemma:linearBehavior} there exists $\alpha > 0$ such that $v_{\infty} = \alpha (x \cdot e_N)$ in $\R^N_{+}$, and by continuity it vanishes at $x_N = 0$. 

Now we perform a blow-up for $v_{\infty}$ at the origin. This time we denote $s_n$ and $u_{\infty}$ respectively the positive sequence of radii and the blow-up limit given by Lemma \ref{lemma:BlowUpConverges}. We claim that
\[
u_{\infty}(x) = \alpha (x \cdot e_N) \, 1_{\{ x_N \ge 0 \}}
\]
To verify the claim, let us denote $\mathcal{A} \coloneqq \Omega^+(v_{\infty}) \cap \{ X_N < 0 \}$. We distinguish two mutually exclusive cases.
\begin{itemize}
    \item[i)] $0 \notin \overline{\mathcal{A}}$. This case is trivial since, for $n$ large enough, $B_{s_n} \cap \Omega^+(v_\infty) \subset \{ x_N \ge 0 \}$.
    \item[ii)] $0 \in \overline{\mathcal{A}}$. Since $\mathcal{A}$ has the outer ball condition, by \cite[Lemma 11.17]{CaffarelliSalsa:GeomApproachToFreeBoundary} there exists $\beta \ge 0$ such that
    \[
    u_{\infty}(x) = \alpha (x \cdot e_N) \, 1_{\{ x_N \ge 0 \}} - \beta (x \cdot e_N) \, 1_{\{ x_N < 0 \}}.
    \]
    However, necessarily $\beta = 0$. This is, for instance, a consequence of the density estimates, and can be obtained easily by contradiction comparing $u_{\infty}(x)$ in $B_1(0)$ with its harmonic extension.
\end{itemize}
We are only left to exhibit a positive sequence $r_n \to 0$ such that the convergence properties $i) - iv)$ of Lemma \ref{lemma:BlowUpConverges} of $u_{r_n}$ to $u_{\infty}$ are satisfied. This can be done by a diagonal argument. We only do this for $i)$ since the other points can be dealt with in the same way.

Let $R_j = j \in \N$, $\varepsilon_j = 1/j$. Choose $n = n(j)$ large enough so that
\[
\left \| \frac{1}{s_n} v_{\infty}(s_n x) - u_{\infty}  \right\|_{L^{\infty}(B_{R_j})} \le \frac{1}{2j}.
\]
Given $s_n$, choose $c_n \coloneqq c_n(s_n)$ so that
\[
\left \| \frac{1}{s_n c_n} u(c_n s_n x) - \frac{1}{s_n} v_{\infty}(s_n x)  \right\|_{L^{\infty}(B_{R_j})} = \frac{1}{s_n} \left \| \frac{1}{c_n} u(c_n x) -  v_{\infty}(x)  \right\|_{L^{\infty}(s_n B_{R_j})} \le \frac{1}{2j}.
\]
The desired sequence is given by $r_n = s_n c_n$.\medskip

Finally, the condition $\alpha>0$, follows from the non-degeneracy of $u$, while $\alpha\le 1$ follows from an argument by contradiction. Indeed, if we suppose that $\alpha>1$, then an internal perturbation of $u_\infty$ with a vector field of the form $\xi=-e_N\phi$ (with $\phi\in C^\infty_c(\R^N)$) produces a competitor with strictly lower energy.
\end{proof}

In the process of proving Theorem \ref{thm:NonSharpRegularity}, it will be useful to consider the notion of \emph{viscosity solution} to problem \eqref{eqn:MainPb}, which we recall below.

Let $0 \le v \in C^1(B_1)$ and $x_0 \in \overline{\Omega^+(u)} \cap B_1$.
\begin{itemize}
    \item[i)] We say that a function $\phi \in C(B_1)$ touches $v$ (strictly) from above at $x_0$ if there exists $r > 0$ such that $B_r(x_0) \subset B_1$ and $\phi^+ \ge u$ in $B_r(x_0)$ ($\phi^+ > u$ in $B_r(x_0) \setminus \{ x_0 \}$ strictly).
    \item[ii)] We say that a function $\phi \in C(B_1)$ touches $v$ (strictly) from below at $x_0$ if there exists $r > 0$ such that $B_r(x_0) \subset B_1$ and $\phi \le u$ in $B_r(x_0)$ ($\phi^+ < u$ in $B_r(x_0) \setminus \{ x_0 \}$ strictly).
\end{itemize}
The notion of viscosity solution can be given in terms of strict comparison sub/supersolution.
\begin{itemize}
    \item[i)] We say that a function $\phi \in C(B_1)$ is a strict comparison supersolution in $B_r(x_0) \subset B_1$ if $\phi \in C^1\left( \overline{\Omega^+(\phi)} \right) \cap B_r(x_0)$, $L(\phi) < f$ in $B_r(x_0)$ in the sense of distributions, and $\vert D_a \phi \vert < Q$ in $\partial \{ \phi >0 \}$.
    \item[ii)] We say that a function $\phi \in C(B_1)$ is a strict comparison subsolution in $B_r(x_0) \subset B_1$ if $\phi \in C^1\left( \overline{\Omega^+(\phi)} \right) \cap B_r(x_0)$, $L(\phi) > f$ in $B_r(x_0)$ in the sense of distributions, and $\vert D_a \phi \vert > Q$ in $\partial \{ \phi >0 \}$.
\end{itemize}
Now we are ready to give the definitions of viscosity sub/super/solution for problem \ref{eqn:MainPb}.
\begin{itemize}
    \item[i)] A function $u \in C(B_1)$ is called a viscosity subsolution of problem \ref{eqn:MainPb} if:no strict comparison supersolution can touch $u$ from above in $(\overline{\Omega^+(u)} \setminus \{ x_N = 0 \}) \cap B_1$,
    \item[ii)] A function $u \in C(B_1)$ is called a viscosity supersolution of problem \ref{eqn:MainPb} if:no strict comparison subsolution can touch $u$ from below in $\overline{\Omega^+(u)} \cap B_1$,
    \item[iii)] A function $u \in C(B_1)$ which is both a viscosity subsolution and supersolution for \ref{eqn:MainPb} is called a viscosity solution.
\end{itemize}
As usual, the following result holds.

\begin{lemma}\label{lemma:SolVariationalIsViscous}
If $u$ is a variational solution of problem \eqref{eqn:MainPb}, then it is also a viscosity solution.
\end{lemma}

\begin{proof}
If a strict comparison sub/supersolution touches $u$ at $x_0 \in \{ u>0\}$, then a contradiction arises from the strong maximum principle.

If, on the other hand, a strict comparison sub/supersolution touches $u$ at $x_0 \in \partial \{ u>0\}$, then $\partial \{ u>0 \}$ has an inner/outer ball condition at $x_0$. Hence, a contradiction is reached after performing a blow-up at $x_0$ and then proceeding similarly as in the proof of \cite[Proposition 7.1]{Velichkov:RegularityOnePhaseFreeBd}.

\end{proof}

\section{Partial Harnack inequality}\label{section:BoundaryHarnack}
This section is dedicated to the partial Harnack inequality, namely to the improvement of flatness at fixed scale. We begin with two preliminary lemmas, that will be useful for the subsequent analysis.

For the next two lemmas, we fix two functions $g \in C^2_c\left( \R^{N-1} \right)$ and $\rho\in C^2_c(\R)$ such that:
\begin{itemize}
\item $g\left(x\right) = 1$ for all $x \in B_{1/2}(0)$ and $g\left(x\right) = 0$ for $x \in B_{3/4}^c(0)$;
\item $\rho(x) = 1$ for $x \in [-1/2, 1/2]$ and $\rho(x) = 0$ for $\vert x \vert \ge 3/4$.
\end{itemize}
Consider the family of diffeomorphisms
\[
\Psi_t(x_1, ..., x_N) \coloneqq \left( x_1, ..., x_N + t \varepsilon \rho(x_N)  g(x_1, ...,x_{N-1}) \right), \qquad t \in [0, 1].
\]
Moreover, denote 
\[
\mathcal{A}_t \coloneqq \Psi_t\left( B_1 \cap \{ x_N \ge - \varepsilon \} \right) \quad \text{and} \quad p_t \coloneqq (\Gamma x_N - \varepsilon)^+ \circ \Psi_t^{-1}.
\]

\begin{lemma}\label{lemma:barriers}
For every $\kappa>1$ and $\eta>0$, there are $\eps_0>0$ and $C>0$ such that the following holds. Given $\eps\in(0,\eps_0)$, a symmetric matrix $A$ with $C^{0,\alpha}$ coefficients satisfying 
$$a_{ij}(0)=\delta_{ij}\qquad\text{and}\qquad \|a_{ij}(x)-\delta_{ij}\|_{C^{0,\alpha}(B_1)}<\eta\eps\,,$$
and every constant $\Gamma\in(\frac{1}{\kappa},\kappa)$, 
we have that:
\begin{enumerate}
\item the solution $h$ to the problem
\begin{align*}
  \begin{cases}   Lh = 0 & \text{ in } B_{1}^+\\
   h = \Gamma(x_N - \varepsilon)^+ & \text{ on } \partial B_{1} \cap \{ x_N > 0 \} \\
     h = 0 & \text{ on }  B_{1} \cap \{ x_N = 0 \} ,
    \end{cases}
    \end{align*}
satisfies     
\begin{equation*}
\vert \nabla h - \Gamma e_N \vert \le C \varepsilon \text{ in } \overline{B_{1/2}^+} 
\end{equation*}
\item the solution $H_t$ to the problem
\begin{align*}
  \begin{cases}   LH_t = 0 & \text{ in } \mathcal{A}_t\\
   H_t = p_t & \text{ on }  \partial \mathcal{A}_t ,
    \end{cases}
    \end{align*}
satisfies
\begin{equation*}
\vert \nabla H_t - \Gamma e_N \vert \le C \varepsilon \text{ in } \overline{B_{1/2}^+}
\end{equation*}
uniformly for $t \in [0, 1]$.
\end{enumerate}
\end{lemma}
\begin{proof}
    For the proof of point $1$, see for instance \cite[Lemma 4.2]{MaialeTortoneVelichkov:2021epsilonregularity} or the proof of \cite[Lemma 3.1]{ferreriVelichkov2023:oneSidedTwoPhase}.

    Concerning point $2$, for any $t \in [0, 1]$, by the Schauder estimates 
\[
   \| H_t - p_t \|_{C^{1, \alpha}}\left(\overline{\mathcal{A}_t}\right) \le c \eta \varepsilon \quad \text{and} \quad \vert \nabla p_t - \Gamma e_N \vert \le c \eta \varepsilon \text{ on } \overline{\mathcal{A}_t}
\]
for some constant $c > 0$ independent on $t$. Hence, the claim follows.
\end{proof}

In the following lemma, we prove that if the free boundary is sufficiently flat with the slope of the function sufficiently distant from the imposed one (in our case $\vert D_a u \vert = 1$), then locally the free boundary collapses to the obstacle. Such property, is very much in the spirit of the regularity for the two-phase Bernoulli problem \cite{DePhilippisSpolaorVelichkov2021:TwoPhaseBernoulli}.
\begin{lemma}\label{lemma:solIsFlatAtContact}
Let $u$ be a viscosity solution to problem \eqref{eqn:MainPb}, under the condition \eqref{eqn:SmallnessCond2}. There exist universal positive constants $\varepsilon_0, \delta_0$, $C$ such that, if $\eps\in(0,\eps_0]$ and
\[
(\Gamma x_N - \varepsilon)^+ \le u \le (\Gamma x_N + \varepsilon)^+ \quad \text{and} \quad  \vert 1 - \Gamma \vert > C \varepsilon,
\]
then 
\[
\overline{\Omega^+(u)} \cap \overline{B_{1/2}} = \{ x_N \ge 0 \} \cap \overline{B_{1/2}}\ .
\]
In particular, we have that 
\[
(\Gamma-c\eps)x_N^+\le u(x)\le (\Gamma+c\eps)x_N^+\quad\text{in}\quad B_{1/4}\ ,
\]
for a dimensional constant $c>0$.
\end{lemma}

\begin{proof}
For any $t \in [0, 1]$, define the functions $z_{0, t}, z_{1, t} \in H^1\left(B_1\right)$ such that
\begin{align*}
L z_{0, t} = f \text{ in } \mathcal{A}_t \quad & \text{and} \quad z_{0, t} = p_t \text{ on } \partial \mathcal{A}_t, \\
L z_{1, t} = - \delta < 0 \text{ in } \mathcal{A}_t \quad & \text{and} \quad z_{1, t} = 0 \text{ on } \partial \mathcal{A}_t, \\
z_{0, t} = z_{1, t} & = 0 \text{ on } B_1 \setminus \mathcal{A}_t.
\end{align*}

By the maximum principle it holds that $z_{1, t} \ge 0$ in $\mathcal{A}_t$, and by elliptic regularity (e.g. \cite[Theorem 8.16]{GilbargTrudinger:Elliptic98}) 
\begin{equation}\label{eqn:solIsFlatAtContactContinuity}
\| z_{i, t_0} - z_{i, t} \|_{L^{\infty}\left(B_1\right)} \to 0 \text{ as } t \to t_0, \qquad i = 1, 2.
\end{equation}
Moreover, by the Schauder estimates 
\begin{equation}\label{eqn:solIsFlatAtContactSchauder}
    \| z_{1, t} \|_{C^{1, \alpha}\left(\overline{\mathcal{A}_t}\right)} \le c \quad \text{and} \quad \| z_{0, t} - p_t \|_{C^{1, \alpha}}\left(\overline{\mathcal{A}_t}\right) \le c \varepsilon
\end{equation}
for some constant $c > 0$ independent on $t$. In particular, in $\overline{\mathcal{A}_t}$
\[
\vert D_a z_{0, t} \vert^2 \le (1+c \varepsilon \delta) \vert \nabla z_{0, t} \vert^2 \le (1+c \varepsilon \delta) \left( \vert \nabla p_t \vert^2 + 3 c \varepsilon \right) \le 1- \frac{C}{2} \varepsilon
\]
uniformly in $t$, for $\varepsilon_0, \delta_0$ sufficiently small, and $C >0$ sufficiently large but universal.

Now, for $t \in [0, 1]$ define 
\[
w_t \coloneqq z_{0, t} + 2 \varepsilon \delta z_{1, t} \quad \text{in } \mathcal{A}_t.
\]
By the above properties and the hypotheses, for $\varepsilon_0, \delta_0$ sufficiently small and $C > 0$ sufficiently large
\begin{align}\label{eqn:solISFlatAtContactStrictSupersol}
\begin{split}
& w_0 > u \text{ in } \overline{\mathcal{A}_0}, \\
& L w_t < f \text{ in } \mathcal{A}_t \text{ for all } t \in [0, 1], \\
& \vert D_a w_t \vert < 1 \text{ in } \overline{\mathcal{A}_t} \text{ for all } t \in [0, 1].
\end{split}
\end{align}
Let $t'$ be the largest $t \in [0, 1]$ such that $w_{t} > u \text{ in } \overline{\mathcal{A}_{t'}} \cap \overline{\Omega^+(u)}$. Now we prove by contradiction that $t' = 1$. Indeed, by \eqref{eqn:solIsFlatAtContactContinuity} and \eqref{eqn:solISFlatAtContactStrictSupersol} we have $w_{t'} \ge u$ on $\overline{\mathcal{A}_{t'}} \cap \overline{\Omega^+(u)}$. On the other hand, since $u$ is a viscosity solution to problem \eqref{eqn:MainPb}, \eqref{eqn:SmallnessCond2}, if $t'<1$ by \eqref{eqn:solISFlatAtContactStrictSupersol} it holds $w_{t'} > u$ on $\overline{\mathcal{A}_{t'}} \cap \overline{\Omega^+(u)}$ and by continuity (up to the boundary) $w_{t'} \ge u + \eta$ on the same set, for some $\eta > 0$. Taking again into account \eqref{eqn:solIsFlatAtContactContinuity} and \eqref{eqn:solIsFlatAtContactSchauder} we have reached a contradiction, hence $t' = 1$.

Since $w_t = 0$ on $\partial \mathcal{A}_t \setminus \partial B_1$ for all $t \in [0, 1]$, the fact that $t' = 1$ implies
\[
\overline{\Omega^+(u)} \cap \overline{B_{1/2}} \subset \overline{B_{1/2}^+} \quad \text{and} \quad u = 0 \text{ on } \overline{B_{1/2}^+} \cap \{x_N = 0\},
\]
which is the desired result if $1-\Gamma > C  \varepsilon$.

If, on the other hand, $1-\Gamma < - C \varepsilon$, one can argue similarly as before and build a family of (moving) barriers $v_t$ which are strict comparison subsolutions for u for all $t \in [0, 1]$ satisfying $\vert D_a v_t \vert \ge 1 + C \varepsilon/2$. By construction, $v_0$ touches $u$ from below at $x = 0$, which contradicts the condition on $\vert D_a u \vert$. This implies that, if $\varepsilon_0, \delta_0$ are sufficiently small, only the first case can occur.
\end{proof}

As a consequence of Lemma \ref{lemma:solIsFlatAtContact}, we have the following

\begin{corollary}\label{crl:NormDeriv1AtDetach}
    Let $u$ be a variational solution to problem \eqref{eqn:MainPb}, \eqref{eqn:SmallnessCond1}, \eqref{eqn:SmallnessCond2}, and suppose that the origin is a detachment point for the free boundary. Then, $\partial_N u(0) = 1$.
\end{corollary}
\begin{proof}
    To begin with, $\partial_N u(0)$ exists by Lemma \ref{lemma:linearBehavior}. Then, the statement follows by contradiction. Indeed, if $\partial_N u(0) \neq 1$, by Lemma \ref{lemma:SolVariationalIsViscous} and Corollary \ref{crl:BlowUpHalfPlane} after a suitable scaling the hypotheses of Lemma \ref{lemma:solIsFlatAtContact} would be satisfied, thus reaching a contradiction.
\end{proof}

When dealing with points of contact with the obstacle, one key observation is that the proof of Harnack's inequality for the unconstrained free boundary in \cite{DeSilva:FreeBdRegularityOnePhase} (or \cite{DeSilvaFerrariSalsa:2PhaseFreeBdDivergenceForm} for the case of elliptic operators in divergence form) still holds provided the free boundary of the barrier is not allowed to touch the obstacle. This, in turn, can be translated in a condition on the flatness of the solution.

As in \cite{ChangLaraSavin:BoundaryRegularityOnePhase} we begin with the following

\begin{lemma}\label{lemma:ImprovFlatnFixScaleOneSide}
Let $u$ be a viscosity solution to \eqref{eqn:MainPb}, \eqref{eqn:SmallnessCond2}, $\sigma \in \R$, $\varepsilon > 0$ be constants such that $ \sigma > - \varepsilon/2$ and

\[
(x_N + \sigma)^{+} \le u \le (x_N + \sigma + \varepsilon)^{+} \quad \text{in } B_1\ .
\]
There exist universal constants $\varepsilon_0 > 0$ and $0<\theta<1$ such that, if $ a+b < \varepsilon_0$, then in $B_{1/20}$
\[
\left(x_N + \sigma + \theta \varepsilon/2\right)^{+} \le u \quad \text{or} \quad u \le \left(x_N + \sigma + ( 1 - \theta/2) \varepsilon \right)^{+}.
\]
\end{lemma}

\begin{proof}
If $\sigma > 1/10$ the result follows from the interior Harnack inequality (see e.g. \cite{DeSilva:FreeBdRegularityOnePhase, Velichkov:RegularityOnePhaseFreeBd}). Hence, without loss of generality we can assume $\sigma \le 1/10$.

The improvement of the flatness from below can be dealt with similarly as in \cite{DeSilva:FreeBdRegularityOnePhase, DeSilvaFerrariSalsa:2PhaseFreeBdDivergenceForm}, and does not require the condition $\sigma \ge -\varepsilon/2$. For this reason, we just report the proof of the improvement from above, and at the end we highlight the main differences and analogies.

Let $p(x) \coloneqq x_N + \sigma$, $\overline{x} \coloneqq 1/5 e_N$ and let us distinguish two cases: 

\[
p(\overline{x}) + \varepsilon/2  \ge u(\overline{x}) \quad \text{and} \quad p(\overline{x}) + \varepsilon/2  < u(\overline{x})
\]

Let us begin with the hypothesis $p(\overline{x}) + \varepsilon/2  \ge u(\overline{x})$. The thesis is proved via a barrier argument. First of all, notice that by the hypotheses $p(x) + \varepsilon - u \ge 0$ in $B_1^+$. Moreover, since $u > 0$ in $B_{1/10}\left(\overline{x}\right)$, it holds that

\[
L(p(x) + \varepsilon - u) = \diverg\left((A-I)e_N\right) - f \quad \text{in } B_{1/10}\left(\overline{x}\right),
\]
with $\| (A-I)e_N \|_{L^{\infty}(\overline{B_1})} + \| f \|_{L^{\infty}(\overline{B_1})} \le \varepsilon \delta$ by \eqref{eqn:SmallnessCond2}. Hence, applying the interior Harnack inequality (e.g. \cite[Theorems 8.17 and 8.18]{GilbargTrudinger:Elliptic98}) we deduce that

\[
p(x) + \varepsilon - u \ge c \frac{\varepsilon}{2} - C \varepsilon \delta \quad \text{in } \overline{B_{1/20}}\left(\overline{x}\right)
\]
for universal constants $c, C > 0$. Thus, for $\varepsilon_0, \delta_0$ sufficiently small it holds that

\begin{equation}\label{eqn:BdHarnackBoundBelow}
p(x) + \varepsilon - u \ge c \frac{\varepsilon}{2} \quad \text{in } \overline{B_{1/20}}\left(\overline{x}\right),
\end{equation}
for some (small) universal constant $c>0$.

Now introduce the function

\[
w(x) \coloneqq
\begin{cases}
1 & \text{in } B_{1/20}\left(\overline{x}\right), \\
C (x - \overline{x})^{-\gamma} - (3/4)^{-\gamma} & \text{in } \overline{B_{3/4}}\left(\overline{x}\right) \setminus B_{1/20}\left(\overline{x}\right),
\end{cases}
\]
where $C>0$ is chosen such that $w$ is continuous, and $\gamma>0 \in \N$ is fixed sufficiently large so that $\Delta w > 0$ in $B_{3/4}(\overline{x}) \setminus B_{1/20}\left(\overline{x}\right)$.

As in \cite{DeSilvaFerrariSalsa:2PhaseFreeBdDivergenceForm}, we build the barrier 
\[
\Psi_t \coloneqq v_t + \phi_t
\]
as a sum of two parts: $v_t$ (which includes the function $w$) with strictly positive laplacian in $B_{3/4}(\overline{x}) \setminus B_{1/20}\left(\overline{x}\right)$, and $\phi_t$ which allows to transfer an analogous (but with opposite sign) property for the operator $L$ to $\Psi_t$, namely $L\Psi_t < 0$ in $B_{3/4}(\overline{x}) \setminus B_{1/20}\left(\overline{x}\right)$. More precisely $v_t$ is defined as
\begin{equation}\label{eqn:BdHarnackDefvt}
v_t \coloneqq p(x) + \varepsilon + c \frac{\varepsilon}{2} (1-w) - c \frac{\varepsilon}{2} t, \quad t \in [0, 1]
\end{equation}
while $\phi_t$ is defined by the problem
\begin{equation}\label{eqn:Barrier3ndPartDef}
\begin{cases}
L\phi_t = - div\left((A-I)\nabla v_t\right) & \text{in } B_{3/4}(\overline{x}), \\
\phi_t = 0 & \text{on }
\partial B_{3/4}(\overline{x}) .
\end{cases}
\end{equation}
By elliptic regularity we have
\begin{equation}\label{eqn:BdHarnackphiC1Unifest}
\| \phi_t \|_{C^{1, \alpha}\left(\overline{B^{3/4}}(\overline{x})\right)} \le C \varepsilon \delta
\end{equation}
uniformly in $t$, for some universal constant $C>0$.

Our aim now is to prove that
\begin{equation}\label{eqn:BdHarnacktoProve}
\Phi_t^+ \ge u \text{ in } \overline{B_{3/4}}(\overline{x}) \text{ for all } t \in [0, 1].
\end{equation}
Indeed, if such condition holds true, choosing $t=1$ and by the definition of $v_t$ we get that
\[
u \le (x_N + \sigma + \varepsilon - c \frac{\varepsilon}{2} w + \Phi_1)^+ \quad \text{in } \overline{B_{3/4}}(\overline{x}),
\]
but since $w>\delta>0$ in $\overline{B_{1/2}}(\overline{x})$, choosing $\varepsilon_0, \delta_0$ sufficiently small it holds that
\[
u \le (x_N + \sigma + \varepsilon - \theta \frac{\varepsilon}{2})^+ \quad \text{in } \overline{B_{1/2}}(\overline{x}),
\]
for some universal constant $0 < \theta \le 1$, which is the thesis.

Now we proceed to prove \eqref{eqn:BdHarnacktoProve} by contradiction. In particular, we will prove a reformulation of \eqref{eqn:BdHarnacktoProve}, that is
\begin{equation}\label{eqn:BdHarnacktoProve2}
u<\Phi_t^+ \text{ in } \overline{B_{3/4}}(\overline{x}) \cap \overline{\Omega^+(u)} \text{ for all } t \in [0, 1).
\end{equation}
To this aim the key condition $\sigma > -\varepsilon/2$ comes into play. Indeed, it guarantees that the barrier is \emph{strictly positive} at $\overline{B_{3/4}}(\overline{x}) \cap \overline{\Omega^+(u)} \cap \{ x_N = 0 \}$ \emph{for any} $t \in [0, 1]$. Namely, the barrier cannot touch $u$ from above at $\{ x_N = 0\}$.

To begin with, we notice that by \eqref{eqn:BdHarnackDefvt} and \eqref{eqn:BdHarnackphiC1Unifest}, for $\varepsilon_0, \delta_0$ sufficiently small it holds that
\[
u \le \Phi_0^+ \quad \text{in } \overline{B^{3/4}}(\overline{x}).
\]
Let $\overline{t} \in [0, 1]$ be the largest $t$ such that $u \le \Phi_t^+$ in $\overline{B_{3/4}}(\overline{x})$ and suppose by contradiction that $\overline{t} < 1$. Let $x'$ be the touching point. If $\overline{t} < 1$, by \eqref{eqn:BdHarnackDefvt}, \eqref{eqn:Barrier3ndPartDef} and the hypotheses we have $u < \Phi_{\overline{t}}^+$ on $\partial B_{3/4}(\overline{x})$, so that that $x' \notin \partial B_{3/4}(\overline{x})$.
Now, a direct computation combined with \eqref{eqn:BdHarnackphiC1Unifest} shows that
\[
\vert D_a \Psi_{\overline{t}} \vert < 1 - c \varepsilon \quad \text{on } \overline{B_{3/4}}(\overline{x}) \cap \{ x_N \le 1/10 \} ,
\]
for some universal constant $c$ and $\varepsilon_0, \delta_0$ sufficiently small. Now, we notice also that

\begin{align*}
& \overline{B_{3/4}}(\overline{x}) \cap \partial\{ u = 0\} \subset \overline{B_{3/4}}(\overline{x}) \cap \{ x_N \le 1/10 \}, \\
& L\Psi_{\overline{t}} < 0 \quad \text{in } B_{3/4}(\overline{x}) \setminus B_{1/20}\left(\overline{x}\right), \\
& \Psi_{\overline{t}} > u = 0 \quad \text{on }\overline{B^{3/4}}(\overline{x}) \cap \overline{\Omega^+(u)} \cap \{ x_N = 0 \} \quad \text{(from the condition $a \ge b$)},
\end{align*}
and recalling that $u$ is a viscosity solution of problem \eqref{eqn:MainPb}, the above properties imply that $x' \notin B_{3/4}(\overline{x}) \setminus B_{1/20}\left(\overline{x}\right) \cap \overline{\Omega^+(u)}$.
Hence, necessarily $x' \in \overline{B_{1/20}}\left(\overline{x}\right)$, but this would lead a contradiction with \eqref{eqn:BdHarnackBoundBelow}, for $\varepsilon_0$ sufficiently small. 
We have just shown that $x' \notin \overline{B_{3/4}}(\overline{x})$. Hence, \eqref{eqn:BdHarnacktoProve2} is proved and so is \eqref{eqn:BdHarnacktoProve}.\medskip

The proof of the case $p(\overline{x}) + \varepsilon/2  \le u(\overline{x})$ can be dealt with in a similar way, considering the following family of barriers from below $\Psi_t \coloneqq v_t + \phi_t$, with

\[
v_t \coloneqq p(x) - c \frac{\varepsilon}{2} (1-w) + c \frac{\varepsilon}{2} t, \quad t \in [0, 1]
\]
and $\phi_t$ defined exactly as in the previous case.

Notice that in this case the condition $\sigma \ge -\varepsilon/2$ is no longer needed, since on the set $\overline{B^{3/4}}(\overline{x}) \cap \overline{\Omega^+(u)} \cap \{ x_N = 0 \}$, the weaker free boundary condition $\vert D_a u \vert \le 1$ actually helps to reach the contradiction.
\end{proof}

%
%



The improvement of flatness at fixed scale stated in Lemma \ref{lemma:ImprovFlatnFixScaleOneSide} does not seem suitable for iteration, or in other words, useful to prove a compactness result. The main reason is the condition $\sigma \ge -\varepsilon/2$, which does not necessarily hold true after an iteration. Nonetheless, reasoning as in \cite{ChangLaraSavin:BoundaryRegularityOnePhase}, such condition can be compensated by a \emph{linear} bound from below for $u$, i.e. a bound which is invariant by scaling. In particular, the bound given by Lemma \ref{lemma:linearBehavior} (ii) can be used to improve Lemma \ref{lemma:ImprovFlatnFixScaleOneSide} into the following partial Harnack result.

\begin{lemma}\label{lemma:ImprovFlatnFixScale}
Let $u$ be a viscosity solution to \eqref{eqn:MainPb}, \eqref{eqn:SmallnessCond2}, $x_0 \in B_{1/2}$ and suppose that
\[
(x_N + \sigma)^+ \le u \le (x_N + \sigma + \varepsilon)^{+} \quad \text{in } B_{1/2}\left( x_0 \right)
\]
for some $\sigma \in \R$ and $\varepsilon > 0$.

There exist universal constants $\varepsilon' > 0$, $\rho > 0$ and $0<\delta<1$ such that, if $ \varepsilon < \varepsilon'$, then in $B_{\rho}\left( x_0 \right)$
\[
(x_N + \sigma + \delta \varepsilon)^+ \le u \quad \text{or} \quad u \le (x_N + \sigma + (1 - \delta) \varepsilon)^{+}.
\]
\end{lemma}
\begin{proof}
In order to give the proof of the lemma, we need to distinguish several cases. To begin with, let us suppose that $x_0 \in \{ x_N = 0 \}$. Without loss of generality, we can further assume that $x_0 = 0$. Under such assumption, we analyze the following two cases.
\begin{itemize}
    \item[i)] $\sigma > - \varepsilon/2$. In this case, the claim is a direct consequence of Lemma \ref{lemma:ImprovFlatnFixScaleOneSide}.

    \item[ii)] $\sigma \le - \varepsilon/2$. In this case one can use the linear bound from below (see Lemma \ref{lemma:barriers})
    \begin{equation}\label{eqn:LinearBoundBelow}
        (1- C \varepsilon) x_N^+ \le u \quad \text{in } B_{1/2}
    \end{equation}
    for some universal constant $C > 0$. Indeed, choosing $\rho$ small enough, e.g. so that $C \rho < 1/4$, we have
    \[
    \left( x_N - 1/4 \varepsilon \right)^+ \le \left( x_N - C \rho \varepsilon \right)^+ \le \left( 1 - C\varepsilon \right) x_N^+ \le u,
    \]
    which gives the claim. 
    %
    %
\end{itemize}
This concludes the proof for the points $x_0 \in \{ x_N = 0 \}$. Suppose now that $x_0 \in \{ x_N < 0 \}$. In this case, we can proceed with a projection argument. More precisely
\begin{itemize}
    \item[iii)] if $\dist{x_0, \{ x_N = 0 \}} \ge \rho/2$, up to choosing $\varepsilon$ small enough so that $\varepsilon/\rho < \varepsilon_0$, the claim follows from the partial Harnack inequality for the one phase free boundaries \cite{DeSilva:FreeBdRegularityOnePhase} in $B_{\rho/2}(x_0)$.

    \item[iv)] if, on the other hand, $d \coloneqq \dist{x_0, \{ x_N = 0 \}} < \rho/2$, let $x_1$ be the projection of $x_0$ onto $\{ x_N = 0 \}$. Then, we can use the previous analysis (the one for $x_0 \in \{ x_N = 0 \}$) centered at the point $x_1$, and observe that $B_{\rho/2}(x_0) \subset B_{\rho}(x_1)$.
\end{itemize}
We are left to consider the case $x_0 \in \{ x_N < 0 \}$. Once again, we proceed with a projection argument distignuishing two cases.
\begin{itemize}
    \item[v)] if $\dist{x_0, \{ x_N = 0 \}} \ge \rho/2$, up to choosing $\varepsilon$ small enough (depending only on $\rho$), the claim follows from the interior Harnack inequality for harmonic functions in $B_{\rho/2}(x_0)$.

    \item[vi)] if $d \coloneqq \dist{x_0, \{ x_N = 0 \}} < \rho/2$ we can proceed exactly as in point $iv)$.
\end{itemize}
%








\end{proof}


%
%

\section{Improvement of flatness at the obstacle}\label{section:ImprovFlatness}
In this section we exploit the previous Lemma \ref{lemma:ImprovFlatnFixScale} in order to prove the following improvement of flatness result at the points of the obstacle.
\begin{proposition}\label{prop:improvFlatness}
Let $u$ be a viscosity solution to \eqref{eqn:MainPb}, \eqref{eqn:SmallnessCond2} with $0 \in \partial \Omega^+(u)$. Suppose that
\[
\left( \alpha x_N - \varepsilon \right)^{+} \le u \le \left( \alpha x_N + \varepsilon \right)^{+} \quad \text{in } B_1
\]
for some $\varepsilon > 0$ .

Then, for all $0 < \beta < \min\{\alpha, 1/2\}$ there exist constants $\varepsilon_0, \delta_0, c, \rho > 0$ depending on $\beta$ but independent of $u$ such that, if $ \varepsilon < \varepsilon_0$ and $\delta < \delta_0$, then in $B_{\rho}$
\[
\left(\alpha' x_N - \varepsilon \rho^{1+\beta}\right)^+ \le u \le \left(\alpha' x_N + \varepsilon \rho^{1+\beta}\right)^{+},
\]
with
\[
\vert \alpha - \alpha' \vert < c \varepsilon.
\]
\end{proposition}

We adapt the idea in \cite{DeSilva:FreeBdRegularityOnePhase}, proceeding by contradiction and studying the behavior, as $\varepsilon \to 0^+$, of the family of normalized perturbations
\[
\tilde{u}_{\varepsilon}(x) \coloneqq \frac{u(x) - \alpha x_N}{\varepsilon} \qquad \text{in } B_{1/2}(y).
\]
By distinguishing two regimes, it is actually sufficient to study the case $\alpha_{\varepsilon} = 1$. Indeed, one can always refer to such case if $\vert 1 - \alpha \vert < C \varepsilon$, since trivially
\begin{equation}\label{eqn:RelTildeuDifferentPts}
        \tilde{u}_{\varepsilon}(x) =  \frac{u(x) - x_N}{\varepsilon} + \frac{1-\alpha}{\varepsilon}x_N  \text{ in } B_{1/2}(y), \text{ with } \left \vert \frac{1-\alpha}{\varepsilon} \right \vert \le C
    \end{equation}
uniformly in $\varepsilon$. On the other hand, in the case $\vert 1 - \alpha \vert > \varepsilon$, by Lemma \ref{lemma:solIsFlatAtContact} the free boundary is flat, hence we can apply the improvement of flatness for L-harmonic functions on a half-ball.

Similarly as in \cite{DeSilva:FreeBdRegularityOnePhase, DeSilvaFerrariSalsa:2PhaseFreeBdDivergenceForm} the proof for $\alpha = 1$ can be subdivided into three main steps: compactness, study of the limit problem, and improvement by contradiction. We proceed in order.

\subsection{Compactness}
An important consequence of Lemma \ref{lemma:ImprovFlatnFixScale} is the following

\begin{corollary}\label{crl:C1aOutsideBall}
Let the quantities $u(x), \varepsilon, \varepsilon', \rho$ and $\delta$ be as in Lemma \ref{lemma:ImprovFlatnFixScale}. Then, for any point $y \in \overline{B_{1/2}}$

\[
\vert \tilde{u}(x) - \tilde{u}(y) \vert \le C \vert x \vert^{\gamma} \text{ for } x \in \overline{\Omega^+(u)} \cap (B_{1/2}(y) \setminus B_{\varepsilon/\varepsilon'}(y)) ,
\]

for some universal constant $C>0$, and $\gamma$ defined by the relation

\[
\rho^{\gamma} = (1-\delta).
\]

\end{corollary}

\begin{proof}
For the sake of simplicity, we give the proof in the case $y = 0$, since the general case follows analogously.

Let $\overline{N} \in \N$ be the largest natural number satisfying $\varepsilon/\rho^{\overline{N}} \le \varepsilon'$. For each $i = 0, ..., N$ we can iterate Lemma \ref{lemma:ImprovFlatnFixScale} to deduce that

\[
(x_N - b_i)^+ \le u(x) \le (x_N + a_i)^+ \quad \text{in } B_{\rho^i},
\]

where $a_i, b_i \ge 0$ and $a_i + b_i \le (1-\delta)^i \varepsilon$. Thus, in $\overline{\Omega^+(u)} \cap (B_{\rho^i} \setminus B_{\rho^{(i+1)}})$

\[
\frac{\vert \tilde{u}(x) - \tilde{u}(0) \vert}{\vert x \vert^{\gamma}} \le \frac{2}{\rho^{\gamma}}\left(\frac{(1-\delta)}{\rho^{\gamma}}\right)^i,
\]

which gives the desired claim with $C \coloneqq 2/\rho^{\gamma}$.

\end{proof}

The compactness result is contained in the following

\begin{lemma}\label{lemma:compactnessImprovFlatness}
Let $\{u_k\}_{k \in \N}$ be a sequence of viscosity solutions to problem \eqref{eqn:MainPb}, \eqref{eqn:SmallnessCond2}. Suppose that
\[
(x_N - \varepsilon_k)^+ \le u(x) \le (x_N + \varepsilon_k)^+ \text{ in } B_1,
\]
with $\varepsilon_k \to 0$ for $k \to +\infty$. Then, there exists a function $\overline{u} \in C^{0, \beta}\left(\overline{B_{1/2}^+}\right)$ for all $0 < \beta < \gamma$, such that

\begin{itemize}

\item[(i)] $\tilde{u}_k \to \overline{u}$ uniformly on $\overline{B_{1/2}^+} \cap \{x_N \ge \delta \}$ for all $\delta > 0$,   

\item[(ii)] in the $R^{N+1}$-Hausdorff distance

\[
\Gamma_k \coloneqq \left\{ (x, \tilde{u}_k(x)) \in \R^{N+1} : x \in \overline{\Omega^+(u_k)} \right\} \to \Gamma \coloneqq \left\{ (x, \overline{u}(x)) \in \R^{N+1} : x \in \overline{B_{1/2}^+} \right\}.
\]

\end{itemize}

\end{lemma}

\begin{proof}
We begin proving $(i)$. Fix $\delta > 0$. First of all, we notice that the sequence $\tilde{u}_k$ is equibounded, i.e.
\[
-1 \le \tilde{u}_k \le 1 \quad \text{in } B_1.
\]
Hence, the uniform convergence on $\overline{B_{1/2}^+} \cap \{x_N \le \delta \}$ is a consequence of the Schauder estimates and the compactness in Hölder spaces. 

Since the modulus of Hölder continuity given by Corollary \ref{crl:C1aOutsideBall} is independent of $\delta > 0$, the limit function $\overline{u}$ can be extended to a Hölder $\beta$-continuous function up to the boundary, for all $0 < \beta < \gamma$.

Now we turn to $(ii)$. By Corollary \ref{crl:C1aOutsideBall}, for any $x \in \overline{B_{1/2}} \cap \{ x_N \le \delta \}$ such that $(x, \tilde{u}_k(x)) \in \Gamma_k$ there exists $y \in \overline{B_{1/2}} \cap \{x_N > \delta \}$ such that
\begin{equation}\label{eqn:compactnessPartTwo}
\vert (x, \tilde{u}_k(x)) -  \vert (y, \tilde{u}_k(y)) \vert \le \left( \frac{\varepsilon_k}{\varepsilon_0} + \delta + C \left(\frac{\varepsilon_k}{\varepsilon_0} + \delta\right)^{\gamma} \right)
\end{equation}
for some universal constant $C > 0$. If $x \in \overline{B_{1/2}} \cap \{x_N > \delta \}$ one can simply choose $x = y$. Hence,
\[
\dist((x, \overline{u}(x)), \Gamma) \le \vert (x, \tilde{u}_k(x)) -  (y, \tilde{u}_k(y)) \vert + \vert (y, \tilde{u}_k(y)) -  (y, \overline{u}(y)) \vert \to 0
\]
for $k \to +\infty$, uniformly on $\Gamma_k$ by point $(i)$, \eqref{eqn:compactnessPartTwo} and choosing $\delta \to 0$.

A similar argument gives that
\[
\dist((x, \overline{u}(x)), \Gamma_k) \le \vert (x, \overline{u}(x)) - (x, \tilde{u}_k(x)) \vert \to 0, \quad \text{for } k \to +\infty
\]
uniformly for $x \in \left(\overline{B_{1/2}^+}\right)$.
\end{proof}

\begin{remark}\label{rmk:C1aRegularityLimitFunction}
    Actually, by the Schauder estimates and the compactness in Hölder spaces, there exists $\alpha>0$ such that the convergence to $\overline{u}$ holds in $C^{1, \alpha}_{loc}\left(B_{1/2}^+\right)$.
\end{remark}

\begin{remark}\label{rmk:aSmaller1Compactness}
An analogous compactness property holds even if the solution $u$ satisfies the weaker bounds
\[
\left(\alpha_k x_N - \varepsilon_k \right)^+ \le u \le \left(\alpha_k x_N + \varepsilon_k \right)^{+},
\]
with $\vert 1 - \alpha_k \vert \le C \varepsilon_k$ for some constant $C > 0$ independent of $k \in \N$, as long as $C \varepsilon_k \le \varepsilon'$ uniformly in $k \in \N$ ($\varepsilon'$ is the one given by Lemma \ref{lemma:ImprovFlatnFixScale}). 

This follows from the trivial bounds
\[
\left(x_N - (C+1)\varepsilon_k \right)^+ \le u \le \left(x_N + (C+1)\varepsilon_k \right)^{+},
\]
which allow to use Lemma \ref{lemma:ImprovFlatnFixScale} and Corollary \ref{crl:C1aOutsideBall} with $\varepsilon = C \varepsilon_k$.
\end{remark}

\subsection{Limit problem}
A fundamental step in the proof of the improvement of flatness, is to recognize the limit problem solved by the limit function $\overline{u}$ introduced in Lemma \ref{lemma:compactnessImprovFlatness}. Intuitively, its gradient at the origin describes the rate of error committed during a blow-up procedure when a chosen and fixed direction $\nu$ is considered (for our purposes $\nu = e_N$).

For this reason, as in \cite{ChangLaraSavin:BoundaryRegularityOnePhase} let us introduce the Signorini (or thin obstacle) problem.
\begin{equation}\label{eqn:SignoriniPb}
\begin{cases}
\Delta u = 0 & \text{in } B_{1/2}^+ , \\
u \ge 0 & \text{on } \overline{B_{1/2}^+} \cap \{ x_N = 0 \} , \\
\partial_{\nu} u \le 0 & \text{on } \overline{B_{1/2}^+} \cap \{ x_N = 0 \} , \\
\partial_{\nu} u = 0 & \text{on } \overline{B_{1/2}^+} \cap \{ x_N = 0 \} \cap \{ u > 0 \} ,
\end{cases}
\end{equation}
where $\nu$ denotes the inner normal to $\partial B_{1/2}^+$. We say that any non-negative function $u \in C^0 \left( \overline{B_{1/2}^+} \right)$ is a \emph{viscosity subsolution} to problem \eqref{eqn:SignoriniPb} if, for any quadratic polynomial $P$ touching $u$ from below at $x \in \overline{B_{1/2}^+}$ it holds that
\begin{align*}
    & \Delta P \le 0 \quad \text{if } x \in B_{1/2}^+, \\
    & \partial_{\nu} P \le 0 \quad \text{if } x \in \overline{B_{1/2}^+} \cap \{ x_N = 0 \} .
\end{align*}
On the other hand, we say that $u$ is a \emph{viscosity supersolution} to \eqref{eqn:SignoriniPb} if, for any quadratic polynomial $P$ such that $P^+$ touches $u$ from above at $x \in \overline{B_{1/2}^+}$, it holds that
\begin{align*}
    & \Delta P \ge 0 \quad \text{if } x \in x \in B_{1/2}^+, \\
    & \partial_{\nu} P \ge 0 \quad \text{if } x \in \overline{B_{1/2}^+} \cap \{ x_N = 0 \} \cap \{ u > 0 \} ,
\end{align*}
and no condition is imposed where $u$ vanishes at $\{ x_N = 0 \}$. We say that a function $u \in C^0 \left( \overline{B_{1/2}^+} \right)$ is a \emph{viscosity solution} to problem \eqref{eqn:SignoriniPb} if it is both a viscosity subsolution and supersolution. We recall that any variational solution (which coincides with the unique viscosity solution) $u$ to problem \eqref{eqn:SignoriniPb} is $C^{1, 1/2}$ regular, as proved by Athanasopoulos and Caffarelli in \cite{AthanasopoulosCaffarelli:SignoriniPb}.

The main content of this section is the following
\begin{lemma}\label{lemma:limitProblem}
The limit function $\overline{u}$ introduced in Lemma \ref{lemma:compactnessImprovFlatness} is a viscosity solution of problem \eqref{eqn:SignoriniPb}.
\end{lemma}
\begin{proof}
To begin with, we notice that, since the functions $\tilde{u}_k$ solve in $H^1(B_{1/2}^+)$
\[
L_k \tilde{u}_k = \frac{1}{\varepsilon_k} \left[ \diverg\left( (A - I) e_N \right) + f \right] \to 0 \quad \text{by \eqref{eqn:SmallnessCond2}},
\]
thanks to the local $C^{1, \beta}$ convergence stated in Remark \ref{rmk:C1aRegularityLimitFunction} it holds that
\[
\Delta u = 0 \quad \text{in } H^1, \text{ locally in } B_{1/2}^+.
\]
As a consequence, by interior elliptic regularity theory $u \in C^{\infty}_{\loc}(B_{1/2}^+)$ and $\Delta u = 0$ in the classical sense.

Now we prove the first claim. Thanks to the regularity just proved for $u$, it sufficies to check the behavior of the touching polynomials at $\overline{B_{1/2}^+} \cap \{ x_N = 0 \}$. Moreover, it is sufficient to check the definitions of sub/supersolutions for touching polynomials $P$ with $\Delta P > 0$ or $\Delta P < 0$ and touching strictly from below/above, respectively. Indeed, if for instance $P$ is a polynomial touching $\overline{u}$ from below at $\overline{x}$, then
\[
\tilde{P}(x) \coloneqq P(x) - \eta \vert x - \overline{x} \vert^2 - \eta (x_N - \overline{x}_N) + C(\eta) (x_N - \overline{x}_N)^2
\]
touches $\overline{u}$ strictly from below and satisfies $\Delta \tilde{P}(x) > 0$, for some $\eta>0$ sufficiently small and $C(\eta)>0$ sufficiently large depending on $\eta$. The case of polynomials touching from above can be treated analogously. Then, the results for the normal derivative of $P(x)$ will follow letting $\eta \to 0$.

We now prove that $\overline{u}$ is a subsolution to \eqref{eqn:SignoriniPb}. Let $P$ be a quadratic polynomial touching $\overline{u}$ from below at $\overline{x} \in \overline{B_{1/2}^+} \cap \{ x_N = 0 \}$. From Lemma \ref{lemma:compactnessImprovFlatness} (ii) we can deduce the existence of points $\overline{x}_k \in \overline{\Omega^+(u_k)}$, $\overline{x}_k \to \overline{x}$ and of constants $c_k \to 0$ such that $P_k \coloneqq P + c_k$ touches $\tilde{u}_k$ from below at $\overline{x}_k$. In order to exploit this property, we need to transfer the properties of $\Delta P$ to $LP$. To this aim, let us introduce the functions $\tilde{\phi}_k$ defined as
\begin{equation}\label{eqn:PolynomialSubstiteSubsol}
\begin{cases}
L_k \tilde{P}_k = \Delta P_k & \text{in } B_{r}\left(\overline{x}\right), \\
\tilde{P}_k = P_k & \text{on } \partial B_{r}\left(\overline{x}\right),
\end{cases}
\end{equation}
where $r>0$ sufficiently small but fixed. By the Schauder estimates (up to the boundary) and \eqref{eqn:SmallnessCond2}, 
\begin{equation}\label{eqn:LimitPbSchauderEst}
\| \tilde{P}_k - P_k \|_{C^{1, \alpha}\left(\overline{B_r}\right)} \le C \varepsilon_k \delta_k
\end{equation}
for some universal constant $C$. Up to other constants $c'_k \to 0$, the functions $\tilde{P}_k + c_k'$ touch $\tilde{u}_k$ from below at some points $\overline{x}_k' \in \overline{\Omega^+(u_k)}$ satisfying $\overline{x}_k' \to \overline{x}$. Hence, the functions $\phi_k \coloneqq \varepsilon_k (\tilde{P}_k + c'_k) + x_N$ touch $u_k$ from below at $\overline{x}_k'$.

Now, since by definition $L_k \tilde{P}_k = \Delta P_k > 0$ and the functions $u_k$ are viscosity solutions to \eqref{eqn:MainPb}, it holds that $\overline{x}_k' \in \overline{\Omega^+(u_k)} \cap \{ u_k = 0 \}$. This in turn implies that $\vert D_{a_k} \tilde{P}_k \vert^2 \le 1 + \varepsilon_k \delta_k$ at $\overline{x}_k'$. Moreover, \eqref{eqn:SmallnessCond2} and \eqref{eqn:LimitPbSchauderEst} imply that $\left\vert \vert D_{a_k} \tilde{P}_k \vert^2 - \vert \nabla P_k \vert^2 \right\vert \le C \varepsilon_k \delta_k$. Hence,
\[
1 + 2 \varepsilon_k \partial_{X_N} P + o(\varepsilon_k) \le 1 + C \varepsilon_k \delta_k \quad \text{at } \overline{x}_k',
\]
which implies the claim, that is
\begin{align*}
\partial_{\nu} P(\overline{x}) \le 0 .&\qedhere
\end{align*}
\end{proof}

\begin{remark}\label{rmk:LimitNormDerVanishAtDetach}
    If the origin is a detachment point for the free boundary, then $\partial_{\nu} \overline{u}(0) = 0$ in the viscosity sense. Indeed, by Lemma \ref{lemma:limitProblem} it is sufficient to prove that $\partial_{\nu} \overline{u}(0) \ge 0$. To this aim, the above arguments can be modified as follows. 
    
    Instead of solving problem \eqref{eqn:PolynomialSubstiteSubsol} over $B_r$, it can be solved over $B_r^+$. Then, a competitor over the whole ball can be obtained simply applying an extension theorem for $C^{1, \alpha}$ functions (e.g. \cite[Chapter VI]{Stein1970:SingularIntegralsAndDifferentiability}). Notice that the $C^{1, \alpha}$ norm of the extension is controlled by the one inside $B_r^+$, thus it is uniformly bounded in $k \in \N$. This change allows the competitors to be \emph{stricly positive} in $(B_r \cap \{ x_N = 0\}) \setminus \{ 0 \}$. Then the statement follows similarly as above, taking into account Corollary \ref{crl:NormDeriv1AtDetach}.
\end{remark}

\subsection{Improvement}

Following \cite{DeSilva:FreeBdRegularityOnePhase, ChangLaraSavin:BoundaryRegularityOnePhase, DeSilvaFerrariSalsa:2PhaseFreeBdDivergenceForm}, as a consequence of Lemmas \ref{lemma:compactnessImprovFlatness} and \ref{lemma:limitProblem} we can complete the proof of Proposition \ref{prop:improvFlatness}.

\begin{proof}[Proof of Proposition \ref{prop:improvFlatness}.]
Let us proceed by contradiction. Consider a sequence of functions $u_k$ solutions of \eqref{eqn:MainPb}, \eqref{eqn:SmallnessCond2} with $\varepsilon_k, \delta_k \to 0$ and 
\[
\left( \alpha_k x_N - \varepsilon_k \right)^{+} \le u \le \left( \alpha_k x_N + \varepsilon_k \right)^{+} \quad \text{in } B_1,
\]
but for which the thesis does not hold. Then, consider two cases.
\begin{itemize}
    \item[(i)] There exists a subsequence such that $\vert 1 - \alpha_k \vert \le C \varepsilon_k$. In such case, define the functions
    \[
    \tilde{u}_k \coloneqq \frac{u_k - x_N}{\varepsilon_k}.
    \]
    By \eqref{eqn:RelTildeuDifferentPts}, Remark \ref{rmk:aSmaller1Compactness} and Lemma \ref{lemma:limitProblem} there exists a uniform and Hausdorff limit $\overline{u}$, with $\| \overline{u} \|_{L^{\infty}\left( \overline{B_{1/2}^+} \right)} \le 2$ and solving the Signorini problem \eqref{eqn:SignoriniPb}. Hence, by \cite{AthanasopoulosCaffarelli:SignoriniPb} we have $\| \overline{u} \|_{C^{1, 1/2}\left( \overline{B_{1/2}^+} \right)} \le C$ for some universal constant $C > 0$.

    \item[(ii)] There exists a subsequence such that $\vert 1 - \alpha_k \vert > C \varepsilon_k$. Define the functions
    \[
    \tilde{u}_k \coloneqq \frac{u_k - \alpha_k x_N}{\varepsilon_k}.
    \]
    In this case, Lemma \ref{lemma:solIsFlatAtContact} ensures the existence of a $C^{1, \alpha}\left( \overline{B_{1/2}^+} \right)$ limit $\overline{u}$, which is also a Hausdorff limit in the sense of Lemma \ref{lemma:compactnessImprovFlatness} (ii). Moreover, $\| \overline{u} \|_{C^{1, \alpha}\left( \overline{B_{1/2}^+} \right)} \le C$ for some universal constant $C > 0$.
\end{itemize}
Hence, for the sequence $\tilde{u}_k$ there exists a Hausdorff limit $\overline{u} \in C^{1, \alpha}\left( \overline{B_{1/2}^+} \right)$ with $\overline{u}(0)=0$, $u \ge 0$ on $\{ x_N = 0 \} \cap \overline{B_{1/2}^+}$ and $\| \overline{u} \|_{C^{1, \alpha}\left( \overline{B_{1/2}^+} \right)} \le C$ for some universal constant $C > 0$ for all $0 < \alpha \le 1/2$. In particular
\begin{equation}\label{eqn:ImprovFlatEqn1}
\| \overline{u}(x) - \nabla \overline{u}(0) \cdot x \|_{L^{\infty}\left( \overline{B_r^+} \right)} \le C \vert x \vert^{1+\alpha} \qquad \text{in } \overline{B_{1/2}^+}.
\end{equation}
Let us denote $\nabla \overline{u}(0) = \overline{\alpha} e_N$, with $\vert\overline{\alpha}\vert \le C$. Let $\delta > 0$ arbitrarily small but fixed and choose $r > 0$ sufficiently small such that $C \, r^{\delta} \le 1/2$. Then, for instance in the second case, \eqref{eqn:ImprovFlatEqn1} and the Hausdorff convergence imply that, for $k$ sufficiently large
\[
\overline{\alpha} x_N - r^{1+\alpha-\delta} \le \frac{u_k - \alpha_k x_N}{\varepsilon_k} \le \overline{\alpha} x_N + r^{1+\alpha-\delta} \qquad \text{in } \overline{B_r} \cap \overline{\Omega^+(u_k)}.
\]
Hence,
\[
(\alpha' x_N - \varepsilon_k r^{1 + \alpha-\delta})^+ \le u_k \le (\alpha' x_N - \varepsilon_k r^{1 + \alpha-\delta})^+ \quad \text{in } \overline{B_r},
\]
with $\alpha' \coloneqq \alpha_k + \varepsilon_k \overline{\alpha}$, so that $\vert \alpha' - \alpha_k \vert \le C \varepsilon_k$. We have reached a contradiction. In the first case one has that
\[
(\alpha' x_N - \varepsilon_k r^{1 + \alpha-\delta})^+ \le u_k \le (\alpha' x_N - \varepsilon_k r^{1 + \alpha-\delta})^+ \quad \text{in } \overline{B_r},
\]
with $\alpha' \coloneqq 1 + \varepsilon_k \overline{\alpha}$, so that $\vert \alpha' - \alpha_k \vert \le (1+C) \varepsilon_k$. Again a contradiction.
\end{proof}
\begin{remark}
   The above proof also highlights that the threshold for the regularity of the free boundary is a combination between the sharp regularity for the Signorini problem, namely $C^{1, 1/2}$, and the regularity of the obstacle, which we assume to be $C^{1, \alpha}$. More precisely, if $\alpha \le 1/2$ then the decay to the blow-up will be (at least) of order $r^{1+\alpha - \delta}$, while if $\alpha > 1/2$ it will be of order $r^{3/2 - \delta}$, for any $0 < \delta < 1/2$.
   
   The above result is also almost optimal, in the sense that if $\alpha \le 1/2$ it is immediate to build examples in which the free boundary is $C^{1, \alpha}$ regular and no more (it is sufficient that the free boundary sticks to the obstacle), while in the case $\alpha > 1/2$ it is possible, similarly as in \cite{ChangLaraSavin:BoundaryRegularityOnePhase}, to exhibit examples of free boundaries which are $C^{1, 1/2}$ regular and no more.  
\end{remark}
%
%
%
%

%
%

\section{Regularity for the constrained one-phase problem. Proof of Theorem \ref{thm:NonSharpRegularity}}\label{section:Regularity}

This section is devoted to the proof of Theorem \ref{thm:NonSharpRegularity}, which is actually a consequence of the existence and rate of decay of the blow-up at each point of the free boundary, near points where the free boundary is flat. Such information can be obtained combining Proposition \ref{prop:improvFlatness} with the partial Harnack inequality for the one-phase Bernoulli problem (see for instance \cite{DeSilva:FreeBdRegularityOnePhase, Velichkov:RegularityOnePhaseFreeBd}), and is the content of the following

\begin{corollary}\label{crl:RateConvergBlowUp}
Let $u$ be a viscosity solution to problem \eqref{eqn:MainPb}, \eqref{eqn:SmallnessCond1} and \eqref{eqn:SmallnessCond2}, with $0 \in \overline{\Omega^+(u)}$. There exist universal constants $\varepsilon_0, \delta_0, \beta, C >0$ such that, if 
\[
(\alpha x_N - \varepsilon)^+ \le u \le (\alpha x_N + \varepsilon)^+ ,
\]
for some $\alpha \in \R$, then the following holds: for all $x_0 \in \partial\{ u>0\} \cap B_{1/4}$ there exists  $\overline{\alpha} \in \R$ and $\nu \in S^{N-1}$ (depending on $x_0$) such that 
\begin{align*}
\| u_{r} - \overline{\alpha} \nu & x_N^+ \|_{L^{\infty}\left(B_1(x_0)\right)} \le C r^{\beta} \quad \text{for all } r \in (0, 1/4), \\
& \vert \alpha - \overline{\alpha} \vert + \vert \nu - e_N \vert \le C \varepsilon.
\end{align*}
\end{corollary}

\begin{proof}
    If $x_0 \in \{ x_N = 0 \}$, the proof consists in an iteration of Proposition \ref{prop:improvFlatness}, once one notices that the sequence of $\alpha_k$ generated at each step $k \in \N$ is Cauchy, and that the error $\vert \overline{\alpha} - \alpha_k \vert$ on the slope is of the same order of the translation error $\varepsilon_k \coloneqq \varepsilon \rho^{k\beta}$, where $\beta$ is the one given by Proposition \ref{prop:improvFlatness}. In particular, in this case $\nu = e_N$.

    Now let us consider the case $x_0 \in \{ x_N < 0 \}$. Thanks to Lemma \ref{lemma:solIsFlatAtContact}, we can assume without loss of generality that $1 - \alpha \le C \varepsilon$. We proceed by a projection argument, and distinguish two regimes:
    \begin{itemize}
        \item[i)] if $\dist\left(x_0, \partial \{ u>0 \} \cap \{ x_N = 0 \} \right) \ge \rho/2$, where $\rho$ is as in Proposition \ref{prop:improvFlatness}, then we iterate the improvement of flatness for the one-phase Bernoulli problem (e.g. \cite{DeSilva:FreeBdRegularityOnePhase, Velichkov:RegularityOnePhaseFreeBd});
        %
        %

        \item[ii)] if, on the other hand, $\dist\left(x_0, \partial \{ u>0 \} \cap \{ x_N = 0 \} \right) < \rho/2$, let $x_1$ be the projection of $x_0$ onto $\partial \{ u>0 \} \cap \{ x_N = 0 \}$. We can iterate at $x_1$ Proposition \ref{prop:improvFlatness} $k$ times, until
        \[
        \frac{1}{r^k}\dist\left(x_0, \partial \{ u>0 \} \cap \{ x_N = 0 \} \right) < \rho/2.
        \]
    \end{itemize}
\end{proof}
\begin{proof}[\bf Proof of Theorem \ref{thm:NonSharpRegularity}.]
Once Corollary \ref{crl:RateConvergBlowUp} is established, Theorem \ref{thm:NonSharpRegularity} follows by a standard argument, which can be found for instance in \cite{Velichkov:RegularityOnePhaseFreeBd}.
\end{proof}

\section{Epsilon-regularity in the case of discontinuous weight. Proof of Theorem \ref{t:main-main}}\label{section:main-main}

We first notice that by the classical theory for the one-phase problem (see \cite{AltCaffarelli:OnePhaseFreeBd}), we have that the solution $u$ is Lipschitz and nondegenerate. Next, by a change of coordinates, we suppose that
$E=\{x_N>0\}$ and that 
\[
Q(x) \coloneqq Q_1(x)\chi_{\{ x_N > 0 \}} + Q_2(x) \chi_{\{ x_N \le 0 \}}, 
\]
where $Q_1$ and $Q_2$ are $C^{0,\alpha}$ functions in $B_1$ still satisfying the hypothesis
\begin{equation}\label{e:bounds-on-Q-section-regularity}
1-\delta\le Q_1(x)\le 1\quad\text{and}\quad 2\le Q_2(x)\le 2+\delta\quad\text{in}\quad B_1,\end{equation}
with $\delta$ small enough. The function $u$ is now flat in the direction $e_N$. Precisely,
\begin{equation}\label{e:definition-eps-flat-section-regularity}
\gamma (x_N-\eps)^+\le u(x)\le \gamma (x_N+\eps)^+\quad\text{for every}\quad x\in B_1,
\end{equation}
for some constant $\gamma$ such that 
\begin{equation}\label{e:condition-gamma-section-regularity}
1-\delta\le \gamma\le 2+\delta.\end{equation}
 
Finally, $u$ minimizes the functional
$$J_Q(u):=\int_{B_1}\Big(\nabla u\cdot A(x)\nabla u+Q(x)\ind{\{u>0\}}\Big)\,dx,$$
where $A(x)=\big(a^{ij}(x)\big)_{i,j}$ is a symmetric matrix with  $C^{0,\alpha}$ coefficients such that $a^{ij}(0) = \delta^{ij}$.\\ 
In the half-balls $B_1^+$ and $B_1^-$ the function $u$ is minimizer respectively of 
$$J_+(u):=\int_{B_1^+}\Big(\nabla u\cdot A(x)\nabla u+Q_1(x)\ind{\{u>0\}}\Big)\,dx,$$
$$J_-(u):=\int_{B_1^-}\Big(\nabla u\cdot A(x)\nabla u+Q_2(x)\ind{\{u>0\}}\Big)\,dx.$$
Thus, it is a solution (in viscosity sense) to: 
\begin{equation}\label{e:viscosity-solution-section-regularity}
\begin{cases}
  \textrm{div}(A(x)\nabla u)=0\quad\text{in}\quad B_1\cap\{u>0\}\\
  \ \ \qquad|D_au|=Q_1\quad\text{on}\quad B_1^+\cap\partial\{u>0\}\\
  \ \ \qquad|D_au|=Q_2\quad\text{on}\quad B_1^-\cap\partial\{u>0\},
\end{cases}
\end{equation}
where $|D_au|:=\sqrt{\nabla u\cdot A(x)\nabla u\,}$ . Now, similarly as for problem \eqref{eqn:MainPb}, we distinguish different regimes. In particular, proceeding as in Lemma \ref{lemma:solIsFlatAtContact}, we have the following
\begin{lemma}\label{lemma:ApplQDiscSolFlatAtContact}
There are $\eps_0>0$, $\delta_0>0$ and $C>0$ such that the following holds.\\
Let $\eps\in(0,\eps_0]$ and $\delta\in(0,\delta_0]$, and let $u:B_1\to\R$ be a continuous function satisfying \eqref{e:viscosity-solution-section-regularity} with $Q_1$ and $Q_2$ satisfying \eqref{e:bounds-on-Q-section-regularity}. Suppose that $u$ satisfies also the flatness condition \eqref{e:definition-eps-flat-section-regularity} with $\gamma$ as in \eqref{e:condition-gamma-section-regularity} and $\eps>0$. Then:
\begin{itemize}
    \item[\rm(i)] if $\vert 1 - \gamma \vert \le C \varepsilon$, then $\overline{\{u>0\}} \cap \overline{B_{1/2}} \cap \{ x_N < 0 \} = \emptyset$, \\

    \item[\rm(ii)] if $\vert \sqrt{2} - \gamma \vert \le C \varepsilon$, then $\overline{B_{1/2}} \cap \{ x_N \ge 0 \} \subseteq \overline{\{u>0\}}$, \\

    \item[\rm(iii)] in all other cases $\overline{\{u>0\}} \cap \overline{B_{1/2}} = \{ x_N \ge 0 \} \cap \overline{B_{1/2}}$. \\
\end{itemize} 
\end{lemma}
\begin{proof}
We fix $C$ as in Lemma \ref{lemma:barriers} and we choose $\eps_0$ and $\delta_0$ small enough. Suppose that 
\begin{equation}\label{e:gamma-bound-in-lemma-regularity-1}
\gamma\le \sqrt 2-C\eps.
\end{equation}
Then, in $B_1$ the function $u(x)$ is bounded from above by $(\sqrt 2-C\eps)(x_N+\eps)_+$ while at the same time, in $B_1^-$ it satisfies (in viscosity sense)  
\begin{equation*}
\begin{cases}
  \textrm{div}(A(x)\nabla u)=0\quad\text{in}\quad B_1^-\cap\{u>0\}\\
  \ \ \qquad|D_au|\ge \sqrt 2\quad\text{on}\quad B_1^-\cap\partial\{u>0\}.
\end{cases}
\end{equation*}
Thus, by Lemma \ref{lemma:barriers}, we obtain (i). \medskip

Suppose next that 
\begin{equation}\label{e:gamma-bound-in-lemma-regularity-2}
\gamma\ge 1+C\eps.
\end{equation}
Then, in $B_1$ the function $u(x)$ is bounded from below by $(1+C\eps)(x_N-\eps)_+$ while at the same time, in $B_1^+$ it satisfies (in viscosity sense)  
\begin{equation*}
\begin{cases}
  \textrm{div}(A(x)\nabla u)=0\quad\text{in}\quad B_1^+\cap\{u>0\}\\
  \ \ \qquad|D_au|\le 1\quad\text{on}\quad B_1^+\cap\partial\{u>0\}.
\end{cases}
\end{equation*}
Thus, using again Lemma \ref{lemma:barriers}, we get (ii). \medskip

Finally, we notice that if we have both \eqref{e:gamma-bound-in-lemma-regularity-1} and \eqref{e:gamma-bound-in-lemma-regularity-2}, then (iii) holds.
\end{proof}

\begin{proof}[\bf Proof of Theorem \ref{t:main-main}]
By the above Lemma \ref{lemma:ApplQDiscSolFlatAtContact}, we have that a solution $u$ to the problem from Theorem \ref{t:main-main} is a minimizer either to $J_+$ with the constraint $\{u>0\}\subset\{x_N>0\}$ or to the functional $J_-$ with the constraint $\{u>0\}\supset\{x_N>0\}$. In the first case, the conclusion follows from \cite{ChangLaraSavin:BoundaryRegularityOnePhase}, while in the second case we apply Theorem \ref{thm:NonSharpRegularity}. 
\end{proof}

\section{Regularity in the case of discontinuous weight. Proof of Theorem \ref{t:main-main-2}}\label{section:main-main-2}
As we already mentioned in the introduction, any variational solution $u$ to \eqref{eqn:VariationalSOlution} is locally Lipschitz continuous and non-degenerate. Up to a change of coordinates we can suppose that $u$ is a variational solutions to
\begin{equation}\label{eqn:PbExample2Viscous}
\begin{cases}
\textrm{div}(A(x)\nabla u) = 0 & \text{in } \{ u>0 \} \cap B_1 , \\
u = g & \text{on } \partial B_1 , \\
\vert D_a u \vert^2 = Q_1 & \text{on } \partial\{ u>0 \} \cap B_1^+ , \\
\vert D_a u \vert^2 = Q_2 & \text{on } \partial\{ u>0 \} \cap B_1^- , \\
Q_1 \le \vert D_a u \vert^2 \le Q_2 & \text{on } \partial\{ u>0 \} \cap \{ x_N = 0 \} \cap B_1 ,
\end{cases}
\end{equation}
with $Q_1$ and $Q_2$ satisfying \eqref{e:bounds-on-Q-section-regularity}. We proceed in several steps.

\noindent{\bf Step 1. Blow-up limits.} For any point $x_0\in\partial\{u>0\}\cap B_1\cap\{x_N=0\}$ we consider the rescalings 
$$u_{r,x_0}(x):=\frac1ru(x_0+rx).$$
Since $u$ is Lipschitz and non-degenerate, we know that every sequence $r_n\to0$ admits a subsequence (still denoted by $r_n$) converging to a non-trivial function
$$u_\infty:\R^N\to\R$$
uniformly in every ball $B_R\subset\R^N$. As usual, we say that $u_\infty$ is a blow-up limit of $u$ at $x_0$. As in Lemma \ref{lemma:BlowUpConverges}, we have that $u_{x_0,r_n}$ converges to $u_\infty$ strongly in $H^1$ and that $\partial\{u_{r_n,x_0}>0\}$ converges to $\partial\{u_{\infty}>0\}$ locally Hausdorff. This entails that any blow-up limit $u_\infty$ is a minimizer of the functional 
$$J_0(u,B_R)=\int_{B_R}\Big(A(x_0)\nabla u\cdot\nabla u+
\big(Q_1(x_0)\ind{\{x_N>0\}}+Q_2(x_0)\ind{\{x_N<0\}}
\big)\ind{\{u>0\}}\Big)\,dx\,,$$
in every ball $B_R$. Moreover, any blow-up limit is $1$-homogeneous. Indeed, up to a change of coordinates, we may assume that $A(x_0)=Id$ and, proceeding as in \cite{Weiss99:PartialRegularityFreeBd}, we have that the Weiss' boundary adjusted energy 
$$W(u,B_r):=\frac1{r^d}J_0(u,B_r)-\frac1{r^{d+1}}\int_{\partial B_r}u^2\,,$$
is (almost-)monotone (see for instance \cite{SpolaorTreyVelichkov2020:RegularityMultiphaseShapeOptimiz}), which gives the homogeneity of $u_\infty$ by a standard argument (see \cite{Weiss99:PartialRegularityFreeBd}). \medskip

\noindent{\bf Step 2. Regular and singular points.} We say that a point $x_0\in\partial\{u>0\}\cap B_1$ is \emph{regular} ($x_0\in\textrm{Reg}(u)$) if there is a blow-up limit $u_\infty$ at $x_0$ of the form 
$$u_\infty(x)=\Gamma(x.\nu)_+$$
for some constant $\Gamma>0$ and some unit vector $\nu\in \R^N$. Otherwise, we will say that $x_0$ is \emph{singular} ($x_0\in\textrm{Sing}(u)$). We next show that when $x_0\in \textrm{Reg}(u)\cap\{x_N=0\}$, then $u_\infty$ can only be of a specific form. Precisely, we have the following:
\begin{lemma}\label{l:regular-blow-ups}
Let $h$ be a blow-up of $u$ at a point $x_0\in\partial\{u>0\}\cap\{x_N=0\}$ such that 
the set $\{h>0\}$ is a half-space. Then, $h$ is of the form $h(x)=\Gamma(x.\nu)_+$, where
$$\frac{Q_1(x_0)}{e_N\cdot A(x_0)e_N}\le \Gamma^2\le \frac{Q_2(x_0)}{e_N\cdot A(x_0)e_N}\qquad\text{and}\qquad \nu=e_N.$$
\end{lemma}
\begin{proof}
Without loss of generality, we can assume that $x_0=0$, $A(x_0)=Id$, $Q_1(x_0)=1$ and $Q_2(x_0)=2$.
Since $h$ minimizes the functional $J_0$, it is a positive harmonic function in a half-space $T:=\{x\in\R^N\ :\ x\cdot\nu>0\}$ and 
\[
\vert \nabla h \vert^2 = 1 \text{ on } \partial T \cap \{ x_N > 0 \} \qquad \text{and} \qquad \vert \nabla h \vert^2 = 2 \text{ on } \partial T \cap \{ x_N < 0 \},
\]
so that the only admissible cases are
\[
T = \{ x_N < 0 \} \qquad \text{or} \qquad T = \{ x_N > 0 \}.
\]
Let us analyze the two cases separately.
\begin{itemize}
    \item[i)] Suppose that $T = \{ x_N < 0 \}$. If so, the function $h$ is also a minimizer of the Alt-Caffarelli's one-phase functional over $\R^N$, with multiplier $\Lambda = 2$. This implies (e.g. \cite[Lemma 6.11]{Velichkov:RegularityOnePhaseFreeBd}) that $\vert \nabla h \vert^2 = 2$ over $\{ x_N = 0 \}$. Hence, by unique continuation, $h(x) = \sqrt{2} x_N^-$. However, this is a contradiction. 
    Indeed, let us consider a family of diffeomorphisms of $B_1$ onto itself
    \[
    \Psi_t(x) \coloneqq x + t \varphi e_N, \text{ with } \varphi \in C^{\infty}_c(B_1),
    \]
    and define $h_t \coloneqq h \circ \Psi^{-1}_t$. Let us compute the first variation for $t = 0$ of the functional $J_0$ 
    \begin{align*}
    \begin{split}
    \frac{d}{dt} \int_{B_1} \vert \nabla h_t \vert^2 = -2 & \int_{B_1} \partial_N h \nabla \varphi \cdot \nabla h + \int_{B_1} \vert \nabla h \vert^2 \partial_N \varphi = 2 \int_{\{ x_N = 0 \}} \varphi, \\
    & \frac{d}{dt} \int_{\{ h_t > 0 \}} Q^2 = - \int_{\{ x_N = 0 \}} \varphi.
    \end{split}
    \end{align*}
    Hence, 
    \[
    \frac{d}{dt} \left[ \int_{B_1} \vert \nabla h_t \vert^2 + \int_{\{ h_t > 0 \}} Q^2 \right] = \int_{\{ x_N = 0 \}} \varphi, 
    \]
    which can be chosen strictly negative deforming $\{ h > 0\}$ in the direction of the smallest value for $Q(x)$, namely choosing $\varphi(x) \ge 0$ but not identically zero on $\{ x_N = 0 \}$.

    \item[ii)] Now suppose that $T = \{ x_N > 0 \}$. By using the first variation as above, we have that $1 \le \vert \nabla h \vert^2 \le 2$ on $\{ x_N = 0 \}$. Moreover, since $h$ is globally Lipschitz, by the Liouville's theorem, we have that $h$ is linear in $T$.\qedhere
\end{itemize}
\end{proof}

\noindent{\bf Step 3. Conclusion of the proof of Theorem \ref{t:main-main-2}.} 
The above Lemma \ref{l:regular-blow-ups} combined with Lemma \ref{lemma:ApplQDiscSolFlatAtContact} give that near regular points on the jump set $\partial E$ the solution $u$ to the variational problem \eqref{e:minimality-J-with-jump}, either the free boundary locally coincides with the discontinuity set, or it reduces to a problem with internal/external inclusion constraint. Thus, the $C^{1,\alpha}$ regularity of $\textrm{Reg}(u)$ follows from Theorem \ref{t:main-main}. The estimate on the dimension of the singular set follows by the Federer's dimension reduction principle (see for instance \cite{Weiss99:PartialRegularityFreeBd,Velichkov:RegularityOnePhaseFreeBd}). Thus, we are only left to show that $N^{\ast\ast}>2$ (that is, $\textrm{Sing}(u)=\emptyset$ in dimension two). This follows from the fact that any blow-up $h$ is a one-homogeneous non-negative function, which is harmonic in its positivity set $\{h>0\}$. In dimension two, this implies that $\{h>0\}$ is a half-space and that $h$ is a half-space solution (see Lemma \ref{l:regular-blow-ups}).

\section{A shape optimization problem with internal constraint}\label{section:applications}


For any $\Omega \subset \R^N$ quasi open set with finite measure, let $\lambda_1 \left( \Omega \right)$ denote the principal eigenvalue of the laplace operator in $\Omega$ with homogeneous dirichlet boundary conditions, namely
\[
\lambda_1 \left( \Omega \right) \coloneqq \inf_{u \in H^1_0(\Omega), u \neq 0} \frac{ \int_{\R^N} \vert \nabla u \vert^2 } { \int_{\R^N} u^2},
\]
which is actually a positive minimum. Let $D \subset \R^N$ be a bounded open set of class $C^{1, \alpha}$ for some $\alpha >0$, and consider the shape optimization problem
\begin{equation}\label{eqn:PbSpaheOptimizExtInclusion}
\lambda_a \coloneqq \inf\left\{ \lambda_1 \left( \Omega \right) : D \subseteq \Omega \subset \R^N, \Omega \text{ quasi open with } \vert \Omega \vert \le a \right\}, 
\end{equation}
where $a > 0$ is a fixed constant and represents a volume contraint. The existence of a minimizer $\Omega'$ for problem \eqref{eqn:PbSpaheOptimizExtInclusion} is proved in \cite{BucurButtazzoVelichkov:ShapeOptimizInternalContraint} via a $\gamma$-convergence and concentration-compactness argument. In particular, $\Omega'$ coincides quasi everywhere with the positivity set of (the precise representative of) the principal eigenfunction $u_{\Omega'}$ associated to $\lambda_1 \left( \Omega' \right)$.

Our aim is to study the regularity of $\partial \Omega'$ up to $\partial D$. As a consequence of Theorem \ref{thm:NonSharpRegularity}, we will show that the following result holds.
\begin{corollary}\label{crl:ShapeOptimizInclusionC1aRegularity}
    There exists a bounded open set $\Omega^{\ast}$ coinciding q.e. with $\Omega'$ such that, for any $\beta \in [0, \min\{\alpha,1/2\})$ the set $\Omega^{\ast}$ is of class $C^{1, \beta}$.
\end{corollary}
Openness and boundedness are given by \cite[Remark 5.3, Proposition 5.12]{BucurButtazzoVelichkov:ShapeOptimizInternalContraint} respectively, while the regularity is a consequence of the identification of \eqref{eqn:PbSpaheOptimizExtInclusion} with a Bernoulli-type free boundary problem. For this, we follow \cite{BrianconLamboley:RegularityShapeFirstEigenLaplacian}.

Consider the functional
\[
J_{\mu}(u) \coloneqq \int_{\R^N} \vert \nabla u \vert^2 - \lambda_a \int_{\R^N} u^2 + \mu \vert \{ u > 0 \} \vert, \qquad u \in H^1(B_R), \mu \in \R^+.
\]
Then, proceeding similarly as in the proof of \cite[Theorem 1.5]{BrianconLamboley:RegularityShapeFirstEigenLaplacian} (only minor modifications are in order to take into account the inclusion constraint, there exist $R>0, h_0>0$ depending on $u_{\Omega^{\ast}}$, $a$ and $x_0 \in \partial \Omega^+\left( u_{\Omega^{\ast}} \right)$ such that, if $B_R$ is centered at $x_0$ and $\Omega^+\left( u_{\Omega^{\ast}} \right) \cap B_R$ does not coincide with $D \cap B_R$ , the optimal eigenfunction $u_{\Omega^{\ast}}$ satisfies
\begin{itemize}
    \item[i)] $J_{\mu^-}(u) \le J_{\mu^-}(v)$ for all $0 \le v \in H_1(\R^N)$ such that $u - v \in H^1_0(B_R)$, $D \subset \Omega^+(v)$ and $\vert \Omega^+(v) \vert \in [a-h_0, a], $\\
    \item[ii)] $J_{\mu^+}(u) \le J_{\mu^+}(v)$ for all $0 \le v \in H_1(\R^N)$ such that $u - v \in H^1_0(B_R)$, $D \subset \Omega^+(v)$ and $\vert \Omega^+(v) \vert \in [a, a+h_0]$,\\
    \item[iii)] $0 <\mu^- \le \mu^+ < +\infty$ and $\lim_{R \to 0} \mu^{\pm} = \Lambda$.
\end{itemize}
Exactly as in \cite[Section 3]{BrianconLamboley:RegularityShapeFirstEigenLaplacian}, this implies the Lipschitz regularity of $u_{\Omega^{\ast}}$ in $B_R$. The non degeneracy of $u_{\Omega^{\ast}}$ in $B_R \setminus D$ can be deduced exactly as in \cite[Lemma 3.1]{BrianconLamboley:RegularityShapeFirstEigenLaplacian}, and it can be extended to the whole $B_R$ similarly as done in the proof of Lemma \ref{lemma:LocalNonDegenerate}. 

We are only left to prove the regularity. This, however, follows at once from Theorem \ref{thm:NonSharpRegularity} after noticing that, by the above properties and locally (after a  flattening of the obstacle $D$), $u_{\Omega^{\ast}}$ is a viscosity solution of problem \ref{eqn:MainPb} with $Q(x) \equiv \Lambda$.

\subsection*{\bf Acknowledgments}
L.F. and B.V. are supported by the European Research Council (ERC), EU Horizon 2020 programme, through the project ERC VAREG - \it Variational approach to the regularity of the free boundaries \rm (No. 853404). L.F. is also a member of INDAM-GNAMPA.


\bibliographystyle{plain}
\bibliography{FreeBoundary_bib.bib}

\begin{thebibliography}{10}

\bibitem{AltCaffarelli:OnePhaseFreeBd}
H.~W. Alt and L.~A. Caffarelli.
\newblock Existence and regularity for a minimum problem with free boundary.
\newblock {\em J. Reine Angew. Math.}, 325:105--144, 1981.

\bibitem{AthanasopoulosCaffarelli:SignoriniPb}
I.~Athanasopoulos and L.~A. Caffarelli.
\newblock Optimal regularity of lower dimensional obstacle problems.
\newblock {\em Zap. Nauchn. Sem. S.-Peterburg. Otdel. Mat. Inst. Steklov.
  (POMI)}, 310(Kraev. Zadachi Mat. Fiz. i Smezh. Vopr. Teor. Funkts. 35
  [34]):49--66, 226, 2004.

\bibitem{BrascoDePhilippsVelichkov2016:SharpFaberKrahn}
Lorenzo Brasco, Guido~De Philippis, and Bozhidar Velichkov.
\newblock {Faber--Krahn inequalities in sharp quantitative form}.
\newblock {\em Duke Mathematical Journal}, 164(9):1777 -- 1831, 2015.

\bibitem{BrianconLamboley:RegularityShapeFirstEigenLaplacian}
Tanguy Brian\c{c}on and Jimmy Lamboley.
\newblock Regularity of the optimal shape for the first eigenvalue of the
  {L}aplacian with volume and inclusion constraints.
\newblock {\em Ann. Inst. H. Poincar\'{e} C Anal. Non Lin\'{e}aire},
  26(4):1149--1163, 2009.

\bibitem{BucurButtazzoVelichkov:ShapeOptimizInternalContraint}
Dorin Bucur, Giuseppe Buttazzo, and Bozhidar Velichkov.
\newblock Spectral optimization problems with internal constraint.
\newblock {\em Ann. Inst. H. Poincar\'{e} C Anal. Non Lin\'{e}aire},
  30(3):477--495, 2013.

\bibitem{CaffarelliSalsa:GeomApproachToFreeBoundary}
Luis Caffarelli and Sandro Salsa.
\newblock {\em A geometric approach to free boundary problems}, volume~68 of
  {\em Graduate Studies in Mathematics}.
\newblock American Mathematical Society, Providence, RI, 2005.

\bibitem{ChangLaraSavin:BoundaryRegularityOnePhase}
H\'{e}ctor Chang-Lara and Ovidiu Savin.
\newblock Boundary regularity for the free boundary in the one-phase problem.
\newblock In {\em New developments in the analysis of nonlocal operators},
  volume 723 of {\em Contemp. Math.}, pages 149--165. Amer. Math. Soc.,
  [Providence], RI, [2019] \copyright 2019.

\bibitem{DePhilippisSpolaorVelichkov2021:TwoPhaseBernoulli}
Guido De~Philippis, Luca Spolaor, and Bozhidar Velichkov.
\newblock Regularity of the free boundary for the two-phase {B}ernoulli
  problem.
\newblock {\em Invent. Math.}, 225(2):347--394, 2021.

\bibitem{DeSilva:FreeBdRegularityOnePhase}
D.~De~Silva.
\newblock Free boundary regularity for a problem with right hand side.
\newblock {\em Interfaces Free Bound.}, 13(2):223--238, 2011.

\bibitem{DeSilvaFerrariSalsa:2PhaseFreeBdDivergenceForm}
Daniela {De Silva}, Fausto Ferrari, and Sandro Salsa.
\newblock Regularity of the free boundary for two-phase problems governed by
  divergence form equations and applications.
\newblock {\em Nonlinear Analysis}, 138:3--30, 2016.
\newblock Nonlinear Partial Differential Equations, in honor of Juan Luis
  Vázquez for his 70th birthday.

\bibitem{ferreriVelichkov2023:oneSidedTwoPhase}
Lorenzo Ferreri and Bozhidar Velichkov.
\newblock A one-sided two phase bernoulli free boundary problem, 2023.

\bibitem{GilbargTrudinger:Elliptic98}
David Gilbarg and Neil~S. Trudinger.
\newblock {\em Elliptic partial differential equations of second order}.
\newblock Classics in Mathematics. Springer-Verlag, Berlin, 2001.
\newblock Reprint of the 1998 edition.

\bibitem{MaialeTortoneVelichkov:2021epsilonregularity}
Francesco~Paolo Maiale, Giorgio Tortone, and Bozhidar Velichkov.
\newblock Epsilon-regularity for the solutions of a free boundary system, 2021.

\bibitem{SpolaorTreyVelichkov2020:RegularityMultiphaseShapeOptimiz}
Luca Spolaor, Baptiste Trey, and Bozhidar Velichkov.
\newblock Free boundary regularity for a multiphase shape optimization problem.
\newblock {\em Comm. Partial Differential Equations}, 45(2):77--108, 2020.

\bibitem{Stein1970:SingularIntegralsAndDifferentiability}
Elias~M. Stein.
\newblock {\em Singular integrals and differentiability properties of
  functions}.
\newblock Princeton Mathematical Series, No. 30. Princeton University Press,
  Princeton, N.J., 1970.

\bibitem{Velichkov:RegularityOnePhaseFreeBd}
Bozhidar Velichkov.
\newblock {\em Regularity of the one-phase free boundaries}, volume~28 of {\em
  Lecture Notes of the Unione Matematica Italiana}.
\newblock Springer, 2023.

\bibitem{Weiss99:PartialRegularityFreeBd}
Georg~Sebastian Weiss.
\newblock Partial regularity for a minimum problem with free boundary.
\newblock {\em J. Geom. Anal.}, 9(2):317--326, 1999.

\end{thebibliography}


\appendix

\medskip
\small
\begin{flushright}
\noindent 
\verb"lorenzo.ferreri@sns.it"\\
Classe di Scienze, Scuola Normale Superiore\\ 
Piazza dei Cavalieri 7, 56126 Pisa (Italy)
\end{flushright}

\begin{flushright}
\noindent 
\verb"bozhidar.velichkov@unipi.it"\\
Dipartimento di Matematica, Università di Pisa\\ 
Largo Bruno Pontecorvo 5, 56127 Pisa (Italy)
\end{flushright}

\end{document}